\documentclass[11pt, oneside]{article}   	
\usepackage{geometry}                		
\usepackage{graphicx}				
\usepackage{mathtools}

\usepackage{amssymb}
\usepackage{amsthm}
\usepackage[backend = bibtex, sorting=none]{biblatex}
\newtheorem{theorem}{Theorem}[section]
\newtheorem{question}{Question}
\newtheorem{conjecture}[theorem]{Conjecture}

\newtheorem{lemma}[theorem]{Lemma}
\newtheorem{proposition}[theorem]{Proposition}
\newtheorem{example}[theorem]{Example}
\newtheorem{corollary}[theorem]{Corollary}

\newtheorem{remk}[theorem]{Remark}

\newenvironment{remark}{\begin{remk}

\begin{normalfont}}{\end{normalfont}
\end{remk}}

\def\FullBox{\hbox{\vrule width 8pt height 8pt depth 0pt}}

\def\qed{\ifmmode\qquad\FullBox\else{\unskip\nobreak\hfil
\penalty50\hskip1em\null\nobreak\hfil\FullBox
\parfillskip=0pt\finalhyphendemerits=0\endgraf}\fi}

\def\qedsketch{\ifmmode\Box\else{\unskip\nobreak\hfil
\penalty50\hskip1em\null\nobreak\hfil$\Box$
\parfillskip=0pt\finalhyphendemerits=0\endgraf}\fi}

\newenvironment{proofsketch}{\begin{trivlist} \item {\bf
Proof sketch:~}}
  {\qedsketch\end{trivlist}}

\title{A hidden signal in Hofstadter's $H$ sequence}
\date{}

\begin{document}

\author{Rodrigo Angelo}
\maketitle

\begin{abstract}

The Hofstadter $H$ sequence is defined by $H(1) = 1$ and $H(n) = n-H(H(H(n-1)))$ for $n > 1$. If $\alpha$ is the real root of $x^3+x=1$ we show that the numbers $\alpha H(n) \mod 1$ are not uniformly distributed on $[0,1]$, but converge to a distribution we believe is continuous but not differentiable. This is motivated by a discovery of Steinerberger, who found a real number with similar behavior for the Ulam sequence.

Our result is related with the fact that a certain sequence defined from the linear recurrence $h_n=h_{n-1}+h_{n-3}$ has the property $\|x h_n\| \rightarrow 0$ precisely for $x \in \mathbb{Z}[\alpha]$, a phenomenon we inquire for general linear recurrent sequences of integers.

\end{abstract}

\section{Introduction}
\

In ``A Hidden signal in the Ulam sequence" \cite{steinerberger2015hidden}, Steinerberger reports an empirical discovery in a classical sequence defined by Ulam. It is $u_1=1, u_2=2$, and then each new term is the smallest number that can be written uniquely as a sum of two distinct earlier terms. These rules generate a sequence that is chaotic and intractable --- little is known about it. But Steinerberger empirically found a real number $\alpha$ for which the distribution of $\alpha u_n \mod 1$ is not uniform in $[0,1]$, instead it looks like this:

\begin{figure}[h]
\centering
  \includegraphics[width=250pt]{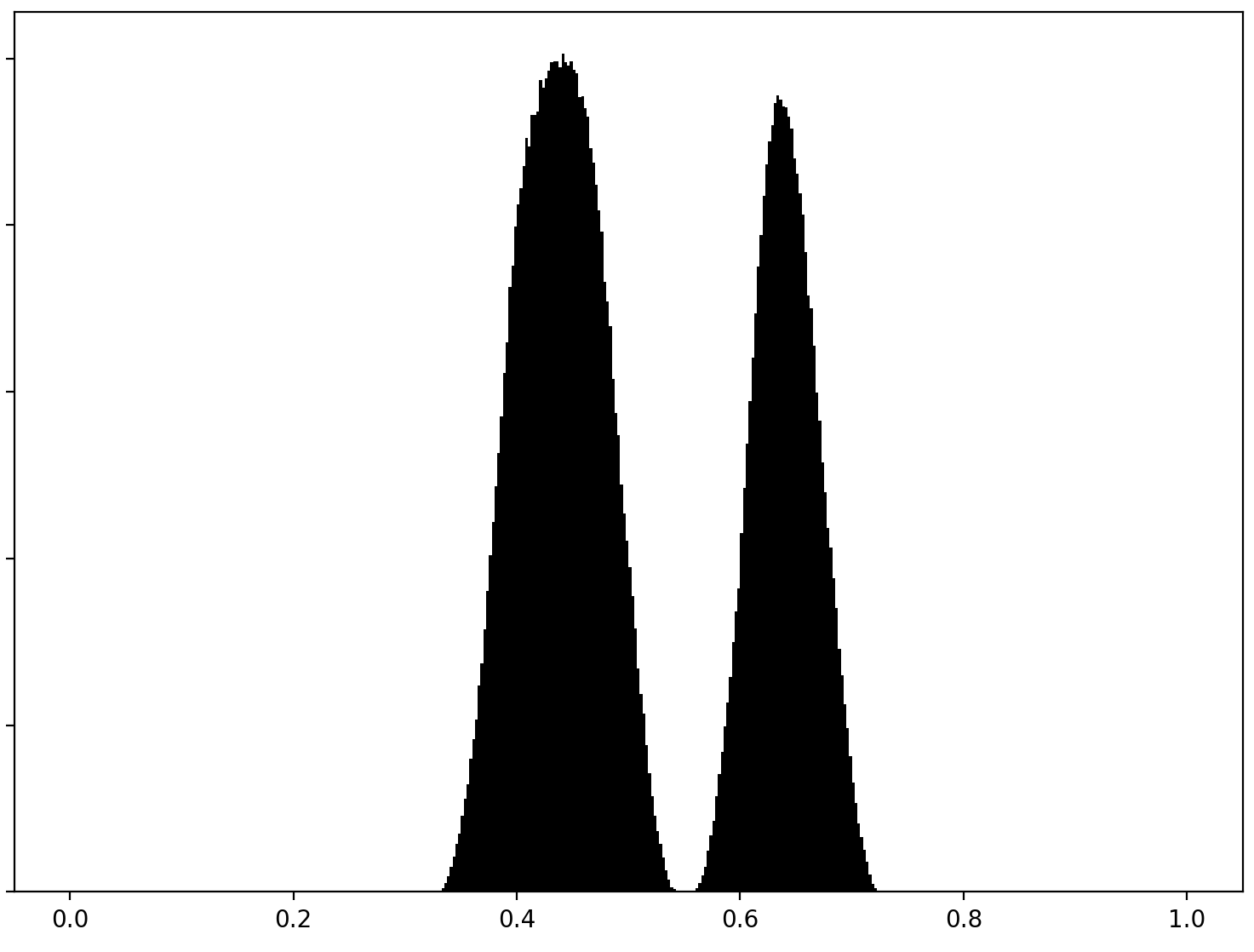}
  \caption{Histogram of $\alpha u_n \mod 1$ for $u_n$ up to $10^6$}
\end{figure}

Here $x \mod 1$ denotes the same as $\{x\}$ - the fractionary part of $x$. This is remarkable because for most naturally arising sequences, such as the natural numbers, the primes or the perfect squares, this distribution is uniform for every irrational number $\alpha$ (see \cite{davenport2000multiplicative, Montgomery1994TenLO}). There is also a theorem of Weyl stating that for any increasing sequence $a_n$ of natural numbers, the distribution of $\alpha a_n \mod 1$ is uniform for almost every real number $\alpha$. 

His method for finding this $\alpha$ was to consider the exponential sums

$$f(x)= \sum_{n \le T} e(x u_n) = \sum_{n \le T} e^{2 \pi i x u_n} $$
and to find a value for $x$ for which $|f(x)| \gg T$ - which means the distribution of $x u_n \mod 1$ can't be uniform, otherwise the terms $e^{2 \pi i x u_n}$ would average out making $f(x) = o(T)$. One reason these sums are convenient to work with is that they can be computed quickly with help of the fast Fourier transform. Then the exceptional value of $\alpha$ looks like an uptick in the plot of the Fourier transform $f(x)$ - hence a ``signal".

He proposes the problem of finding other naturally arising sequences with a real number $\alpha$ such that the distribution of $\alpha a_n$ is not uniform but still absolutely continuous, which he expected to be rare property - this is the question that motivated us.

The Hofstadter $H$ sequence was introduced by Douglas Hofstadter in his beautiful book ``G\"odel, Escher, Bach: An Eternal Golden Braid" \cite{geb}. It is defined by means of a self-referential recursion, with $H(1) = 1$ and

$$H(n) = n - H(H(H(n-1)))$$
for $n > 1$. One can check that it is well-defined and that it grows about linearly, as $H(n) \sim \alpha n$, where $\alpha = 0.682327...$ is the real root of $x^3 + x = 1$. Notice the relationship between this polynomial and the formula defining $H$. Its first few terms are

$$1, 1, 2, 3, 4, 4, 5, 5, 6, 7, 7, 8, 9, 10, 10,...$$

We found that this sequence has a hidden signal at this same root $\alpha$, that is, the distribution of $\alpha H(n) \mod 1$ is not uniform. Rather, this is a picture of it:

\begin{figure}[h]\label{fig:h3fig}
\centering
  \includegraphics[width=250pt]{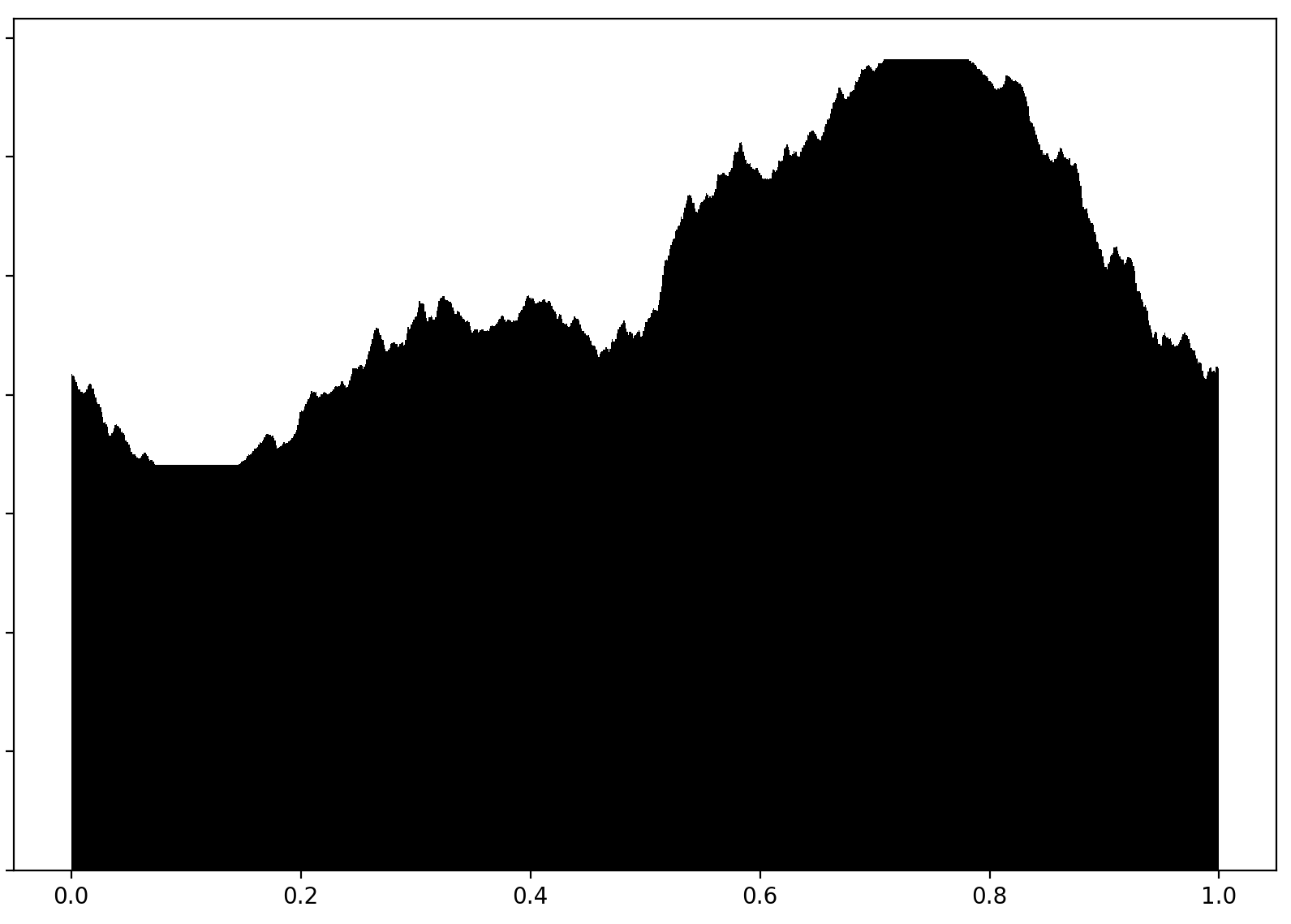}
  \caption{Histogram of $\alpha H(n) \mod 1$ for $n$ up to $10^7$}
\end{figure}

This was found experimentally, looking for spikes in the Fourier transforms of a number of sequences we suspected could carry a hidden signal. Because the $H$ sequence is not nearly as intractable as the Ulam sequence, we manage to explain with proof what is behind this. The following theorem summarizes our results about this sequence:

\begin{theorem} \label{alphadist}

Let $H$ be Hofstadter's $H$ sequence, $\alpha$ be the real root of $x^3+x=1$ and $\beta$ be an irrational number. Then:

\begin{itemize}

\item If $\beta \in \mathbb{Q}[\alpha]$, the distribution of the numbers $\beta H(n) \mod 1$ converges to an absolutely continuous probability measure on the circle.

\item If $\beta \notin \mathbb{Q}[\alpha]$ the distribution of the numbers $\beta H(n) \mod 1$ converges to the uniform distribution in the circle.

\end{itemize}

\end{theorem}

Furthermore, we conjecture that the distribution arising from the first case $\beta \in \mathbb{Q}[\alpha]$ is never uniform, establishing it in the case $\beta = \alpha$ that is pictured above.

We also investigate the $d$-th Hofstadter sequence - defined by the same recursion as $H$, except with $d$ iterations of $H(... H(n-1) ...)$ instead of $3$. Their behavior is somewhat surprising:

\begin{theorem}

Let $H$ denote the $d$-th Hofstadter sequence. For $2\le d \le 5$, the distribution of the numbers $\alpha H(n) \mod 1$ is non-uniform and absolutely continuous, where $\alpha$ is the positive real root of $x^d+x=1$. For $d\ge 6$ however, $\beta H(n) \mod 1$ is uniformly distributed for every irrational $\beta$.

\end{theorem}

In these results the non-uniform behavior of $\beta H(n) \mod 1$  is always associated with an integer multiple of $\beta$ satisfying $\|\beta h_n \| \rightarrow 0$ for an attached linear recurrent sequence $\{h_n\}_{n \in \mathbb{N}}$ (where $\|.\|$ denotes distance to closest integer). We think the problem of determining the set of such $\beta$ for a given linear recurrence is independently interesting, and we study it with more generality. Under some assumptions we find that if $a_n$ is a linear recurrence, the set of real numbers $\beta$ with $\|\beta a_n \| \rightarrow 0$ is a ring of algebraic integers, that can be described with information about the Galois action of the characteristic polynomial of $\{a_n\}_{n \in \mathbb{N}}$ on its roots of absolute value bigger than $1$.

\
\

\section{Outline}
\

In section \ref{a little} we work out a closed formula for $H(n)$ related to Narayana's sequence $h_n$ and a characterization of the real numbers $\beta$ such that $\| \beta h_n \| \rightarrow 0$. We sketch how these properties may be used to establish the hidden signals.

Section 4 is technical - we formalize the work of the previous section, arriving at a criterion regarding the set of signals of a sequence with a similar closed formula. We apply this to the $H$ sequence to characterize its signals and to prove that a limit distribution of $x H(n) \mod 1$ exists at each real number $x$. In section 5 we overview additional applications of this technical criterion.

In section 6 we prove that the distributions on the circle arising from the signals of the $H$ sequence are in fact absolutely continuous. In section 7 we zoom in the picture of the distribution arising from the signal at $x = \alpha$ of the $H$ sequence, and explain some of its idiosyncratic properties.

In section 8 we characterize the signals of the $d$-th generalization of the $H$ sequence. This requires a study of the general problem of given a linear recurrence $a_n$ determining the numbers $\beta$ such that $\| \beta a_n \| \rightarrow 0$.

\section{A formula for $H(n)$} \label{a little}
\

In this section we work out the structure in the sequence $\{H(n)\}_{n \in \mathbb{N}}$ that will explain its signal.

Crucially, it has a closed formula in terms of the Narayana sequence, defined by $h_1 = 1, h_2 = 2, h_3 = 3$ and

$$h_k = h_{k-1}+h_{k-3}$$
for $k > 3$. This is a twist of the Fibonacci sequence originally defined by Narayana Pandita. Its first handful of terms:

$$1,2,3,4,9,13,19,28,41,60,...$$

For any natural number $n$ there is a canonical way of decomposing it as a sum of Narayana numbers (we shall call this the Narayana base representation of $n$). It is obtained by repeatedly subtracting the largest Narayana number below $n$ from it, until nothing is left. The Narayana numbers we obtain from this process are not only distinct, but at least $3$ indices apart from each other. This is because if $h_k \le n < h_{k+1}$, we will subtract $h_k$ from $n$ to obtain

$$n-h_k < h_{k+1}-h_{k} = h_{k-2},$$
hence if $h_k$ is used, neither is $h_{k-1}$ or $h_{k-2}$. This notion is analogous to the Zeckendorf representation of an integer, except with Narayana numbers instead of Fibonacci numbers.

Now the closed formula: $H(n)$ is equal to the right shift of the digits of $n$ in this representation ($h_1=1$ gets shifted to $h_0 = 1$). For example, $16 = 13+3 = h_6+h_3$, so $H(16) = h_5+h_2 = 9+2 = 11$. This closed formula can be proved inductively.

In turn, the Narayana sequence $\{h_k\}_{k \in \mathbb{N}}$ has a close relationship to the real root $\alpha$ of $x^3+x=1$:

\begin{lemma}\label{tendszeronara}

The numbers  $\| \alpha h_k \|$ decrease exponentially as $k \rightarrow \infty$. The same still holds if $\alpha$ is replaced by another number in $\mathbb{Z}[\alpha]$. Here $\|.\|$ denotes distance to the closest integer.

\end{lemma}

\begin{proof}

The sequence $\{h_k\}_{k \in \mathbb{N}}$ has the closed formula

$$h_k = c_1 \alpha^{-k} + c_2 \theta^{-k}+c_3 \overline{\theta}^{-k},$$
where $\alpha$, $\theta$ and $\overline{\theta}$ are the roots of $x^3+x=1$. Notice that $|\alpha|$ is less than $1$, while $|\theta|=|\overline{\theta}|>1$. We can cancel out the leading term with

$$\alpha h_k - h_{k-1} = c_2 (\alpha - \theta) \theta^{-k} + c_2 (\alpha - \overline{\theta})\overline{\theta}^{-k},$$
which decreases exponentially. Because $h_{k-1}$ is an integer, this proves $\| \alpha h_k\|$ decreases exponentially, as desired.

For general $\beta \in \mathbb{Z}[\alpha]$, say $\beta = a_0+a_1\alpha+...+a_d \alpha^d$, we can obtain a similar cancellation for

$$\beta h_k -(a_0 h_k + a_1 h_{k-1} + ... + a_d h_{k-d}),$$
making $\beta h_k$ exponentially close to an integer.

\end{proof}

More notably, the sequence $\{h_k\}_{k \in \mathbb{N}}$ does not share this relationship with any number outside $\mathbb{Z}[\alpha]$:

\begin{lemma} If $\beta$ is a real number such that $\| \beta h_k \| \rightarrow 0$, then $\beta \in \mathbb{Z}[\alpha]$.

\end{lemma}

\begin{proof}

Let $h_k'$ be the closest integer to $\beta h_k$. Then

$$(h_k'-\beta h_k) -  (h_{k-1}' - \beta h_{k-1}) - (h_{k-3}' - \beta h_{k-3})$$
tends to zero. And $h_k = h_{k-1} + h_{k-3}$ so

$$h_k' - h_{k-1}' - h_{k-3}'$$
tends to zero. But this number is an integer, so $\{h_k'\}_{k \in \mathbb{N}}$ eventually satisfies the same linear recurrence as $\{h_k\}_{k \in \mathbb{N}}$.

Now as before write

$$h_k = c_1 \alpha^{-k} + c_2 \theta^{-k}+c_3 \overline{\theta}^{-k}.$$

The polynomial $P(x) = (x- \theta)(x - \overline{\theta})$ has coefficients in $\mathbb{Q}[\alpha]$ and we can use it to neutralize $c_2$ and $c_3$

$$h_2 - (\theta+\overline{\theta})h_1 + \theta \overline \theta h_0 = c_1 P(\alpha^{-1}),$$
which implies $c_1 \in \mathbb{Q}[\alpha]$. Similarly, $\{h_k'\}_{k \in \mathbb{N}}$  has the closed formula for large $k$

$$h_k' = c_1' \alpha^{-k} + c_2' \theta^{-k}+c_3' \overline{\theta}^{-k}.$$
where $c_1' \in \mathbb{Q}[\alpha]$.

Now remember $h_k'$ is the closest integer to $\beta h_k$, so

$$\beta = \lim_{k \rightarrow \infty} \frac{h_k'}{h_k} = \frac{c_1'}{c_1} \in \mathbb{Q}[\alpha].$$

To further narrow it down to $\mathbb{Z}[\alpha]$, write $\beta = a_0+a_1 \alpha + a_2 \alpha^2$, with $a_0, a_2, a_2 \in \mathbb{Q}$. As in the proof of lemma \ref{tendszeronara}, the quantity

$$\beta h_k - (a_0 h_k + a_1 h_{k-1} + a_2 h_{k-2})$$
tends to $0$. Tis means $a_0 h_k + a_1 h_{k-1} + a_2 h_{k-2}$ is  close to an integer for large $k$. But this is a rational number with bounded denominator, hence an integer for large $k$. These satisfy the linear recursion of a unit $x_{k-3} = x_{k}-x_{k-1}$, so tracing back the initial values are integers as well. For say $k = 2,3,4$ we get

\begin{center}
\begin{tabular}{ c }
 $a_0 + a_1 + 2a_2 \in \mathbb{Z}$ \\ 
 $a_0 + 2a_1 + 3a_2 \in \mathbb{Z}$ \\  
 $2a_0 + 3a_1 + 4a_2 \in \mathbb{Z},$ 
\end{tabular}
\end{center}
yielding that $a_0, a_1, a_2 \in \mathbb{Z}$, which completes the proof that $\beta = a_0 +a_1 \alpha + a_2 \alpha^2 \in \mathbb{Z}[\alpha]$.

\end{proof}

Let us return to the problem of understanding the distribution of $\beta H(n) \mod 1$. What we mean by the distribution of these numbers is the limit measure of the sequence of measures

$$\mu_n = \frac{1}{n} \sum_{i = 0}^{n-1} \delta(\beta H(i)),$$
where $\delta(x)$ denotes the Dirac probability measure on the circle, whose entire mass is concentrated at the point $x$.

The closed formula helps us understand this limit because it implies $\mu_{h_n}$ satisfies a recurrence, namely

$$\mu_{h_n} = \frac{h_{n-1}}{h_n} \mu_{h_{n-1}} + \frac{h_{n-3}}{h_{n}} S_{\beta h_{n-2} }(\mu_{n-3}),$$
where $S_x$ denotes the right shift of a measure by $x$. 

This follows from breaking down the terms that appear in the left into two ranges: for $k < h_{n-1}$, then the terms $\delta(\beta H(k))$ appear in $\frac{h_{n-1}}{h_n} \mu_{h_{n-1}}$. Otherwise, the Narayana base representation of $k$ contains the summand $h_{n-1}$, say $k = h_{n-1}+l$, where $l < h_{n-3}$. So $H(k) = h_{n-2} + H(l)$, and $\beta H(k)$ breaks down as $\beta h_{n-2} + \beta H(l)$, so these consist in terms of $\frac{h_{n-3}}{h_{n}} \mu_{n-3}$ shifted by $\beta h_{n-2}$.

The next step here is to look at the Fourier coefficients of these measures. By a version of a theorem of Weyl, which we will later state more rigorously, a sequence of probability measures on the circle $\mu_n$ converges to a measure if and only if for each integer $d$ the sequence of Fourier coefficients $\hat \mu_n (d)$ converges. And the limit measure is the measure whose $d$-th Fourier coefficient matches this limit for each $d$, so this limit measure is uniform if and only if the sequence of $d$-th Fourier coefficients converges to zero for each $d\neq 0$. So we are left with understanding these coefficients, which are the so called  Weyl sums of the sequence $\{\beta H(n)\}_{n \in \mathbb{N}}$. So fix an integer $d$ and take $d$-th Fourier coefficients of the recurrence above, so we get a sequence with the recurrence

$$\hat\mu_{h_n}(d) = \frac{h_{n-1}}{h_n} \hat\mu_{h_{n-1}}(d) + \frac{h_{n-3}}{h_{n-3}} e(d \beta h_{n-2}) \hat \mu_{n-3}(d).$$

So it boils down to an analysis of this recurrence. Let us for readability replace  $\hat\mu_{h_n}(d)$ with $x_n$. From here on we will provide a sketch of this analysis - the next section will deal with it more rigorously. The ratio $\frac{h_{n-1}}{h_n}$ quickly converges to $\alpha$, and $\frac{h_{n-3}}{h_n}$ to $\alpha^3$, so the recurrence is approximately

$$x_n = \alpha x_{n-1} + e(d \beta h_{n-2}) \alpha^3 x_{n-3}.$$

The analysis now breaks down into two cases, which explain the relevance of understanding for which $\beta$ one has $\|\beta h_{n}\| \rightarrow 0$
\\

\textbf{Case 1}: $ \| d \beta h_{n-2} \| \rightarrow 0$ 
\\

In this case  $e(d \beta h_{n-2}) \rightarrow 1$ and the recurrence becomes approximately

$$x_n = \alpha x_{n-1} + \alpha^3 x_{n-3},$$
which is guaranteed to have a limit, since it is a linear recurrence with characteristic polynomial $x^3=\alpha x^2+\alpha^3$, which has a root $x=1$ and two other roots with absolute value less than $1$. Such a sequence necessarily has a limit, so the limit of $x_n$ exists in this case.
\\

\textbf{Case 2}: $\| d \beta h_{n-2} \|$ doesn't converge to $0$
\\

This means $e(d \beta h_{n-2})$ is bounded by a constant away from $1$ infinitely often. So the natural triangle inequality

$$|x_n| = |\alpha x_{n-1} + e(d \beta h_{n-2}) \alpha^3 x_{n-3}| \le \alpha|x_{n-1}| + \alpha^3|x_{n-3}|$$
will be far from sharp infinitely often. Because $\alpha+\alpha^3 = 1$, this guarantees an approximately constant factor $<1$ decrease for $|x_n|$ infinitely often. Hence in this case $x_n$ tends to $0$.
\\

In either case we may conclude that for each $d$ the sequence of Fourier coefficients $\hat\mu_{h_n}(d)$ converges, hence the sequence of measures $\mu_{h_n}$ must converge (and because we know the coefficients tend to zero in some cases, we may in some cases conclude the limit measure is uniform).

In the next section we cary this analysis rigorously in a more general set up. Our approach in the next section also overcomes the issue that the limit was only stablished in the subsequence $\mu_{h_n}$.

\section{Writing in one base and reading as another}
\

In this section we will consider ``replacement sequences" $A(n)$ defined in the following manner: start out with two sequences $\{a_i\}_{i \in \mathbb{N}}$ and $\{b_i\}_{i \in \mathbb{N}}$. For each $n$, write it in greedily in ``base $a_i$" (we soon explain exactly what that means) and replace each appearance of an $a_i$ by the respective $b_i$ - the resulting number is defined to be $A(n)$. These generalize $H(n)$, that is the replacement sequence of $a_i = h_i$ by $b_i = h_{i-1}$. We will prove a criterion for deciding whether each real number $\alpha$ is a signal of $A(n)$ - roughly speaking the criteria only depends on $b_i$: if $\|\alpha b_i\|$ converges to $0$,  then $\alpha$ is a signal, whereas if $\|\alpha b_i\|$  doesn't converge to $0$, then $\alpha$ is not a signal.

We first need to lay down certain technical assumptions on the sequence $a_i$ to make our analysis more workable ($b_i$ is not required to satisfy any assumptions).

Let $a_i$ be increasing with $a_1=1$. Consider the process of writing the terms of the sequence itself in base $a_i$: for each $a_n$, greedly write it as a sum of some $a_{n-1}, a_{n-2},...$ always subtracting the largest $a_i\le a_{n-1}$ less than or equal to what is left (possibly more than once if necessary) until we get to zero - this is a kind of ``signature" of the sequence $a_i$ telling us how each $a_i$ relates to previous terms of the same sequence. Assume there exist a constant $L$ such that for each $n$ this process only uses $a_{n-1},..., a_{n-L}$, so it yields an equality

$$a_n = c_1a_{n-1}+...+c_La_{n-L}$$
for some integers $c_i \ge 0$. The coefficients $(c_1,..., c_L)$ are allowed to be different for different values of $n$. Assume also that $a_i$ satisfies a growth condition: there exist $r,s > 1$ such that for each $n$

$$ra_n < a_{n+1} < sa_{n}.$$

We will call a sequence $a_i$ satisfying those assumptions a sequence of ``mild signature". We borrow this use of the word ``signature" from \cite{hamlingnathanwebb}. Some examples of sequences of mild signature are:

\begin{itemize}

\item The sequence $a_i = b^{i-1}$ for $b \ge 2$, for which our notion of ``writing in base $a_i$" gives the traditional arithmetic base $b$. It's signature comes from $a_i = ba_{i-1}$.

\item The sequence $a_i = F_i$ of Fibonacci numbers, which gives the Zeckendorf coding. Lekkerkerker \cite{Lekkerkerker1951VoorstellingVN} was the first to study this coding, and proved that the greedy algorithm gives a representation of $n$ as a sum of non-consecutive Fibonacci numbers, and that this is the unique representation of $n$ with this property.


\item The Narayana sequence $a_i = h_i$, and indeed many other linear recurrent sequences with non-negative coefficients and $a_1 = 1$. The signatures will often come from the recurrence they satisfy - in many cases one can give some criteria under which these representations are unique, which makes these notions of bases even more compelling. These generalizations have been extensively studied, see for instance \cite{GRABNER199425, BUNDER201499, hamlingnathanwebb}. Uniqueness however won't be a concern in this paper - we just deal with the greedy representation.

\item The sequence $q_n$ of denominators of a real number $\theta = [a_1,a_2,...]$ with bounded continued fraction coefficients. This sequence satisfies the recurrence $q_n = a_n q_{n-1}+q_{n-2}$, which explains its mild signature. It also satisfies $\|\theta q_n \| \rightarrow 0$, which is sure to provide further connections.

\end{itemize}

We can now describe the criterion in detail. In both theorems below, $a_i$ is a sequence of mild signature, $b_i$ is any sequence of integers, and $A(n)$ is the replacement of $a_i$ by $b_i$ in the base $a_i$ representation of $n$.

\begin{theorem} \label{replace}

If $\alpha$ is a real number such that $\| \alpha b_i \| \rightarrow 0$ rapidly ($\sum \| \alpha b_i \| < \infty$ suffices), then the sequence of Weyl sums $x_n = \frac{1}{n} \sum_{i = 0}^{n-1} e(\alpha A(i))$ converges to some complex number.
\end{theorem}

This limit is typically not $0$. In contrast:

\begin{theorem}\label{uniform}

If $\beta$ is a real number such that the sequence $\gamma_i = \|\beta b_i \|$ doesn't converge to $0$, then the sequence of Weyl sums $x_n = \frac{1}{n} \sum_{i = 0}^{n-1} e(\beta A(i))$ converges to 0.

\end{theorem}

The main ideas of the proof of these theorems are already in Section 3 - right now we are merely handling the technicalities generally. So we encourage the reader to skip straight to the corollaries at the end of this section on a first reading.

\begin{proof}[Proof of Theorem \ref{replace}]

Similarly to our sketch for the $H$ sequence, these Weyl sums satisfy a recurrence: if $a_k$ is the largest $a_i$ strictly less than $n$, by breaking down the numbers up to $n-1$ between the ones at greater or equal to $a_k$ and the ones less than $a_k$. Writing the ones greater or equal to $a_k$ as $a_k+m$ where $m$ is between $0$ and $n-1-a_k$, and keeping the ones less than $a_k$ as they are we get

$$n x_n = (n-a_k) e(\alpha b_k) x_{n-a_k} + a_k x_{a_k}.$$

Now because $\| \alpha b_k \|$ converges to $0$ fast, $e(d \alpha a_k)$ converges to 1 fast, so this recurrence is similar to a recurrence of the shape

\begin{equation}\label{exact}
x_n = \frac{(n-a_k)}{n} x_{n-a_k} + \frac{a_k}{n} x_{a_k},
\end{equation}
where from now on $a_k$ will be thought of as a function of $n$ - the largest $a_i$ less than $n$.

\begin{lemma}

If a sequence $\{y_n\}_{n \in \mathbb{N}}$ satisfies equation \ref{exact} for every large $n$, it converges.

\end{lemma}

\begin{proof}

We can assume without loss of generality that each $y_n$ is real by dealing separately with their complex and imaginary part. It follows inductively that the sequence $y_n$ is bounded, so let $a = \limsup y_{a_k}$. Notice that when we recursively apply the equation again to the $n-a_k$ term, this unravels the base $a_i$ representation of $n$. So if $n = \sum a_i$ is the sum obtained from this representation (with possibly repeated $a_i$'s), we get the equation

$$y_n = \sum \frac{a_i}{n} y_{a_i}.$$

Therefore for any $\epsilon>0$, the bound $y_n \le (a+\epsilon)$ holds for all large $n$ (since this inequality holds for $y_{a_i}$ for each large $a_i$, and from the growth condition the weights $\frac{a_i}{n}$ coming from small $a_i$ contribute to only a small proportion of the above sum. Here we use again that the $y_n$ are bounded). 

Now let $k$ be a large index with $y_{a_k} \ge (a-\epsilon)$. From the equation

$$y_{a_k} = \frac{(n-a_k)}{n} y_{a_k-a_{k-1}} + \frac{a_k}{n} y_{a_{k-1}}$$
together with $y_{a_k-a_{k-1}} \le (a+\epsilon)$ we get:

$$y_{a_{k-1}} \ge a - \epsilon \left (\frac{2a_{k}-a_{k-1}}{a_{k-1}} \right) \ge a - \epsilon(2s-1) $$
(using the growth condition $a_k < s a_{k-1}$). Apply this argument repeatedly to get $y_{a_{k-i}} \ge a - \epsilon (2s-1)^i $.

Now for any given $\delta > 0$ by choosing $\epsilon < \delta (2s-1)^{-L}$ we can find an index $k$ such that each $a_{k}, a_{k-1},..., a_{k-L}$ is at least $a - \delta$. Now for an arbitrary $n \ge k$ write $a_n$ greedily as a sum of previous $a_i$'s so

$$y_{a_n} = \sum \frac{a_i}{a_n} y_{a_i}.$$

Because $a_i$ has mild signature, each index $i$ that appears in the sum is inbetween $n-1,..., n-L$. From this equation, together with the fact that $\sum \frac{a_i}{a_n} = 1$ and $y_{a_n} \ge a-\delta$ for each $n$ in  $\{k, k-1,..., k-L\}$, it follows inductively that

$$y_{a_n} \ge a-\delta$$
for each $n \ge k$. Since this holds for any $\delta > 0$, we can conclude that $\liminf   y_{a_k} = a$ as well, that is, $\lim_{k \rightarrow \infty} y_{a_k} = a$ . Again, by writing arbitrary $n$ in base $a_i$ so

$$y_n = \sum \frac{a_i}{n} y_{a_i}$$
and using that the sum of the weights of the $a_i$ for small $i$ is negligible, we may conclude $y_n \rightarrow a$ as $n \rightarrow \infty$, as desired.

\end{proof}

That proved, we can go back to the original sequence $x_n$ which doesn't satisfy equation \ref{exact} exacly. Rewrite the recurrence for $x_n$ as:

\begin{equation} \label{drift}
x_n = \frac{(n-a_k)}{n} x_{n-a_k} + \frac{a_k}{n} x_{a_k} + \frac{n-a_k}{n}r_n,
\end{equation}
where

$$r_n = (1-e(\alpha b_k))x_{n-a_k},$$
which satisfies

$$|r_n| \le |(1-e(\alpha b_k))| \le \| \alpha b_k \|.$$

Here we used that $|x_n|\le 1$, which follows from $x_n$ being a Fourier coefficient of a probability measure. Now take some cutoff $R$ and consider an alternative sequence $y_n$, that for $n \le R$ is defined to be equal to $x_n$, and for $n > R$ is defined to satisfy the recurrence 

$$y_n = \frac{(n-a_k)}{n} y_{n-a_k} + \frac{a_k}{n} y_{a_k}$$
exactly. Equation \ref{drift} shows that $x_n$ and $y_n$ will drift apart only by a little bit (by about $r_n$) at each step. Indeed we have

$$|x_n-y_n| \le S_n,$$
where $S_n$ is the supremum of all sums of the shape

$$|r_{l_1}|+...+|r_{l_s}|$$
with $R \le l_1, l_m \le n$ and each $l_{i+1} \ge c l_i$, where $c = \frac{s}{s-1} > 1$. This can be proved inductively - we have

$$|x_n-y_n| \le \frac{n-a_k}{n}|r_n|+\frac{n-a_k}{n}|x_{n-a_k}-y_{n-a_k}| + \frac{a_k}{n}|x_{a_k}-y_{a_k}|$$ 

$$\le \frac{n-a_k}{n}|r_n|+\frac{n-a_k}{n}S_{n-a_k} + \frac{a_k}{n}S_{a_k}$$
from an induction hypothesis. And notice that that $\frac{s}{s-1} (n- a_k) < n$ (otherwise $a_k$ wouldn't be the biggest $a_i$ below $n$), so

$$|r_n|+S_{n-a_k} \le S_n.$$

And also $a_k \le n$ so $S_{a_k} \le S_n$, so we get

$$|x_n-y_n|\le \frac{n-a_k}{n} S_n + \frac{a_k}{n} S_n = S_n,$$
which completes the induction step. Now remember each $r_n$ satisfies $|r_n| \le \| \alpha b_k \|$, and in any $S_n$, each $\| \alpha b_k \|$ appears at most $\frac{rs}{s-1}$ times, since the $b_k$ of an $r_n$ comes by taking the largest $a_k$ below $n$, and the sequence of $n$'s in a list $l_1,l_2,..., l_s$ in $S_n$ grows at least exponentially by a factor of $\frac{s}{s-1}$, and the ratio between each $a_k$ and the next is by $r$. So $S_n$ is bounded by a constant times the tail after $R$ of the infinite sum:

$$\|\alpha b_1\|+\|\alpha b_2\|+...$$

So for any $\epsilon>0$ we can take $R$ large enough so that this constant times this tail less than $\epsilon$, as to get

$$|x_n - y_n| \le S_n < \epsilon$$
for every large $n$. But we proved previously that the sequence $y_n$ obtained from this $L$ has a limit. Because we can do this for any $\epsilon$, this proves that the sequence $x_n$ is a Cauchy sequence, therefore it converges, which completes the proof of Theorem \ref{replace}.

\end{proof}

\begin{proof}[Proof of Theorem \ref{uniform}]

Like in the proof of Theorem $\ref{replace}$ the sequence $x_n$ satisfies the recurrence

$$x_n = \frac{(n-a_k)}{n} e(\beta b_k) x_{n-a_k} + \frac{a_k}{n} x_{a_k}.$$

By taking absolute values and applying the triangle inequality we get

\begin{equation}\label{ineq}
|x_n| \le \frac{n-a_k}{n}|x_{n-a_k}| + \frac{a_k}{n}|x_{a_k}|.
\end{equation}

By recursively reapplying this inequality we get

$$|x_n| \le \sum \frac{a_i}{n} |x_{a_i}|,$$
where $n = \sum a_i$ is the representation of $n$ in base $a_i$. Again the $x_{n}$ are bounded, so let $a = \limsup |x_{a_k}|$. Again from the negligible contribution of small $a_i$ we may deduce that for any $\epsilon > 0$, the inequality $|x_n| \le a+\epsilon$ holds for every large $n$.

Let $n$ be large enough so that $|x_{a_k}| < a+ \epsilon$ for each $k \ge n-L$. By greedily writing $a_n =\sum a_i$ we get

$$|x_{a_n}| \le \sum \frac{a_{i}}{a_n}|x_{a_{i}}|,$$
where each $i$ is from $n-L, ..., n-1$. From this we deduce that if $L$ consecutive $a_i$'s satisfy $|x_{a_i}| < a - \epsilon$ for some $\epsilon > 0$, then the next $a_i$ also does. This would imply $x_{a_i} < a - \epsilon$ every large $i$, which would contradict $a = \limsup |x_{a_k}|$. So between any $L$ consecutive $a_i$'s, there is one with $|x_{a_i}| \ge a - \epsilon$. But also from the inequality

$$|x_{a_{k+1}}| \le \frac{a_{k+1}-a_k}{a_{k+1}}|x_{a_{k+1}-a_k}| + \frac{a_k}{a_{k+1}}|x_{a_k}|,$$
together with the fact that for large $k$ one has $|x_{a_{k+1}-a_k}| \le a + \epsilon$, and the fact that $\frac{a_k}{a_{k+1}} > \frac{1}{s}$, we get that if $|x_{a_{k+1}}| > a - \epsilon$, then $|x_{a_{k}}| > a - s\epsilon$, where $C$ is some absolute constant. We can repeat this a finite number of times to get that if $|x_{a_k}| > a -\epsilon$ then $|x_{a_{k-t}}| > a - \epsilon ts$. But in any range of $L$ consecutive $x_{a_i}$'s there is one with absolute value at least $a - \epsilon$. So for any $\delta > 0$ one has

$$|x_{a_k}| > a - \delta$$
for every large $k$, so in fact $\lim_{k \rightarrow \infty} |x_{a_k}|$ exists and is equal to $a$. Let's then prove $a = 0$.

Up to now we have only used the inequality \ref{ineq} - we will finally go back to the original equation. Assume by contradiction that $a \neq 0$ - the idea is that this will allow us to talk about the arguments of the $x_{a_n}$, who will align in a certain way to force $e(\beta b_k) \rightarrow 1$. First take a look at the equation

$$x_{a_k} = \frac{(n-a_k)}{n} e(\beta b_k) x_{a_{k}-a_{k-1}} + \frac{a_k}{n} x_{a_{k-1}}.$$

It implies that for any $\epsilon >  0$, the arguments of $x_{a_k}$ and $x_{a_{k-1}}$ are within $\epsilon$ of each other as long as $k$ is large, since $|x_{a_k}| \approx |x_{a_{k-1}}|$ while $|\frac{(n-a_k)}{n} e(\beta b_k) x_{a_{k}-a_{k-1}}| \le \frac{(n-a_k)}{n} (a+\epsilon)$, which forces approximate equality in the triangle inequality, which together with $|x_{a_i}|$ being bounded away from $0$ if $a\neq 0$ forces approximately aligned arguments.

By extending this relationship between consecutive terms a finite number of times we get that for any $\epsilon > 0$ the numbers $x_{a_{k}},..., x_{a_{k-L}}$ have arguments within $\epsilon$ of each other for large $k$. Now apply the recursion

$$x_n = \frac{(n-a_k)}{n} e(\beta b_k) x_{n-a_k} + \frac{a_k}{n} x_{a_k}$$
repeatedly, starting out with $n = a_{k}$ and unraveling the base $a_i$ representation of $n = a_k$. On the left side we get $x_{a_{k}}$, and on the right side we get an expression in which the only $x_i$'s that appear are $x_{a_{k-1}},..., x_{a_{k-L}}$. Notice each term in this expression is composed by a positive real number weight $\frac{a_{k-i}}{a_k}$ (these weights add up to $1$), an $x_{k-i}$ (that have absolute value approximately $a$ and similar arguments), and a unit complex number composed of the product of $e(\beta b_{k-i})$'s.

Now notice that the weights $\frac{a_{k-i}}{a_k}$ are bounded away from zero (from the exponential growth condition on $a_i$), so the fact that $x_{a_k}$ is close to the other $a_i$'s from approximate equality in the triangle inequality implies that each of the unit complex numbers is approximately $1$. But notice that one of these complex numbers is $e(\beta b_k)$ (it is always multiplying the ``second" term), so from this one can deduce that $e(\beta b_k) \rightarrow 1$ as $k \rightarrow \infty$. This implies $\| \beta b_k\| \rightarrow 0$, which is false by assumption.

So indeed $a = 0$, that is, $x_{a_k} \rightarrow 0$. Because for each $\epsilon > 0$ one has $|x_n| \le a+\epsilon$ for all large $n$, this implies $x_n \rightarrow 0$. That is, $\hat \mu_n(d) \rightarrow 0$ for each $d \neq 0$, as desired.

\end{proof}

These results shall be applied together with the following version of a theorem of Weyl, which reduces understanding the distribution of a sequence to understanding its Weyl sums.

\begin{theorem}\label{weyl} If $\mu_1, \mu_2,...$ are probability measures on the circle such that for each $d$ the sequence of their d-th Fourier coefficients $\{\hat \mu_n(d) \}_{n \in \mathbb{N}}$ converges, then $\mu_1, \mu_2,...$ converges weakly to some probability measure $\mu$, whose $d$-th Fourier coefficient is equal to the limit of that sequence for each $d$.

In particular, if for each $d \neq 0$ the sequence $\{\hat \mu_n(d) \}_{n \in \mathbb{N}}$ converges to $0$, then $\mu_1, \mu_2,...$ converges to the uniform distribution.

\end{theorem}

\begin{proof}

We need that for any continuous $f$ the sequence

$$\int f d\mu_n$$
converge as $n \rightarrow \infty$. By approximating $f$ uniformly with smooth functions, it suffices to prove this same statement when $f$ is smooth. But a smooth function has a Fourier expansion $f(x) = \sum_{d \in \mathbb{Z}} a_d e(dx)$ where $\sum|a_d| < \infty$, so

$$\int f d\mu_n = \sum a_d \int e(dx) d\mu_n = \sum a_d \hat \mu_n(d).$$

Because each $\hat \mu_n(d)$ converges as $n \rightarrow \infty$, and  $\sum|a_d| < \infty$, the value of this sum also converges as $n \rightarrow \infty$.

The statement that the Fourier coefficients of the limit distribution are the limit of the corresponding coefficients of $\mu_n$ is a particular case of the weak convergence for $f(x) = e(dx)$.

\end{proof}

We may now return to the original $H$ sequence, and conclude that the limit distributions we observed indeed exist:

\begin{corollary}

For any real number $\beta$ the distribution of $\beta H(n) \mod 1$ converges to a measure in the circle.

\end{corollary}

\begin{proof}

From Theorem \ref{weyl}, it suffices to argue that for each $d \in \mathbb{Z}$ the Weyl sums $x_n = \frac{1}{n} \sum_{i = 0}^{n-1} e(d \beta H(i))$ converge. These are the Weyl sums of a replacement sequence, of $a_n = h_n$ by $b_n = h_{n-1}$, so Theorems \ref{replace} and \ref{uniform} apply. For each $d$ either  $d \beta \in \mathbb{Z}[\alpha]$, in which case we argued in section \ref{a little} that $\| d \beta h_i \| \rightarrow 0$ quickly so Theorem \ref{replace} applies and this Weyl sum converges, or $d \beta \notin \mathbb{Z}[\alpha]$, in which case we argued the numbers $\| d \beta h_i \|$ does not converge to zero so Theorem \ref{uniform} applies and this Weyl sum converges to zero. In either case the sequence of Weyl sum converges. This happens for each $d \in \mathbb{Z}$ so the limit measure must exist.

\end{proof}

\begin{corollary}

Let $\alpha$ be the real root of $x^3+x=1$. For any real number $\beta \notin \mathbb{Q}[\alpha]$ the distribution of $\beta H(n) \mod 1$ converges to the uniform measure in the circle.

\end{corollary}

\begin{proof}

Because $\beta \notin \mathbb{Q}[\alpha]$, for each integer d one has $d \beta \notin \mathbb{Z}[\alpha]$, hence $\{ \| d \beta h_i \| \}_{i \in \mathbb{N}}$ does not converge to zero, and Theorem \ref{uniform} always applies. So each Weyl sum converges to zero, which by Theorem \ref{weyl} implies the numbers $\beta H(n)$ are uniformly distributed in the circle.

\end{proof}

We conjecture but can't prove that in contrast to this corollary, for $\beta \in \mathbb{Q}[\alpha]$ the arising distribution is never uniform. Verifying this non-uniformity in any particular case is simple: the same analysis we performed in the case $d\beta \in \mathbb{Z}[\alpha]$, together with the computation of $\hat \mu_n(d \beta)$ for large enough $n$ can show that the limit will not be too far from the computed terms, hence non-zero. This proves a Fourier coefficient of our distribution is non-zero, which stablishes non-uniformity. For every $d\beta \in \mathbb{Z}[\alpha]$ we  tried we established a non-zero Weyl sum limit at $d\beta$ by this method, but we see no general proof these limits are non-zero.

A third corollary follows from the Weyl sums at the rationals.

\begin{corollary}

For each natural number $m$, the sequence $H(n) \mod m$ is uniformly distributed in $\{0,1,..., m-1\}$.

\end{corollary}

\begin{proof}

This is equivalent to a statement abou the distribution of $\beta H(n) \mod 1$ for $\beta = \frac{1}{m}$. If $d$ is not a multiple of $m$, $d \frac{1}{m} \notin \mathbb{Z}[\alpha]$, so the Weyl sums at $d \frac{1}{m}$ converge to $0$, and if $d$ is a multiple of $m$ clearly the Weyl sums at $d \frac{1}{m}$ converge to $1$. These correspond to the Fourier coefficients of the uniform discrete distribution supported at $\frac{0}{m}, \frac{1}{m}, ..., \frac{m-1}{m}$, which by Theorem \ref{weyl} proves our assertion.

A further remark is that one can use a similar argument to show that for $\beta \in \mathbb{Q}[\alpha]$, if $m$ is the smallest natural with $m \beta \in \mathbb{Z}[\alpha]$, the distribution of $\beta H(n)$ consists of $m$ consecutive scaled down copies of the distribution of $(m \beta) H(n)$ (as this the unique distribution whose $md$-th Fourier coefficient equals the $d$-th Fourier coefficient of the original distribution for each $d$, while the remaining Fourier coefficients at non-multiples of $m$ are zero).

\begin{figure}[h]\label{fig:h3fig}
\centering
  \includegraphics[width=250pt]{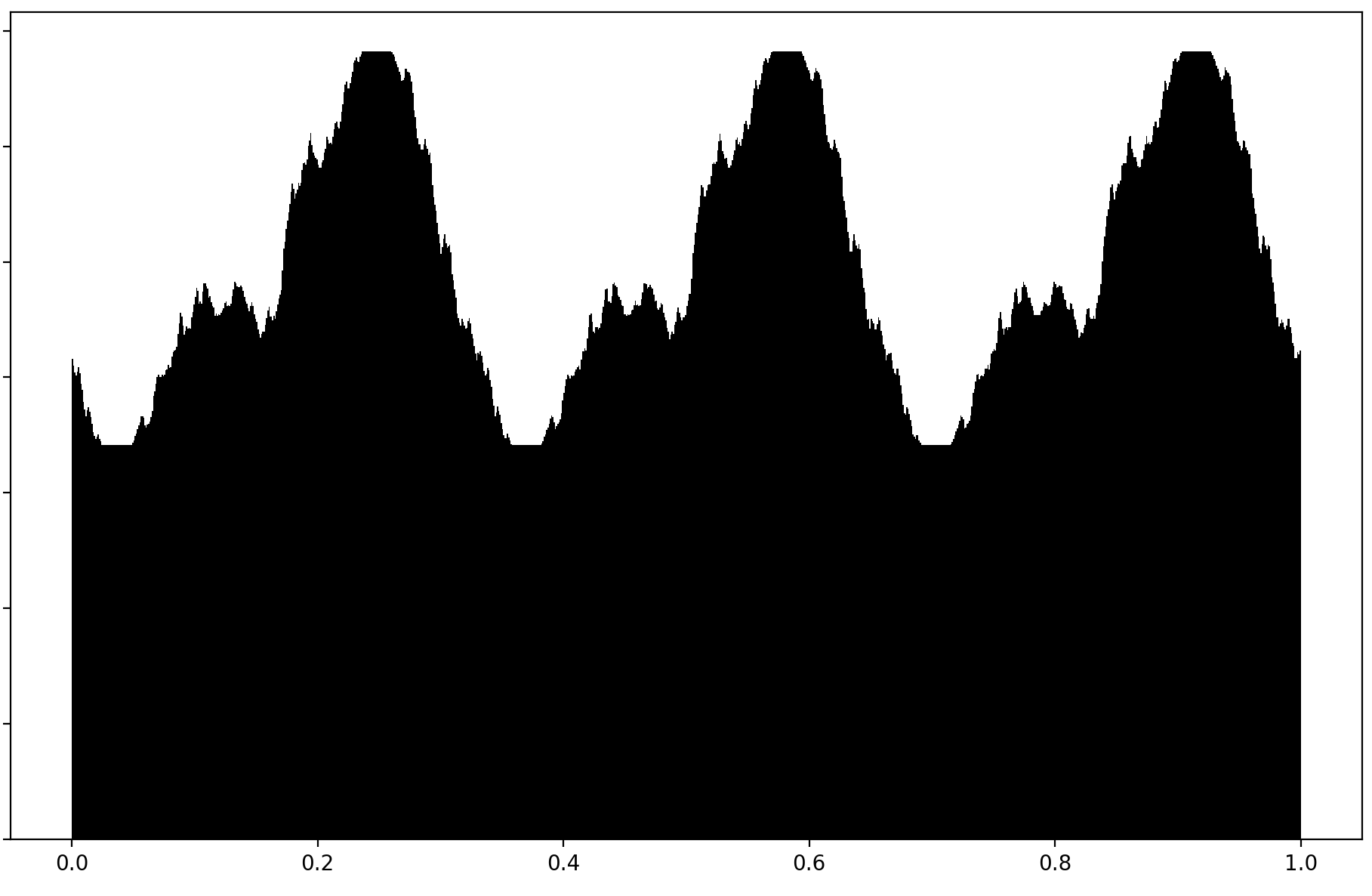}
  \caption{Histogram of $\frac{\alpha}{3} H(n) \mod 1$ for $n$ up to $10^7$}
\end{figure}

\end{proof}

\section{Additional applications}
\

\begin{proposition}

Let $a_i$ be a sequence of mild signature, and let $s(n)$ denote the sum of digits of $n$ when written in base $a_i$. If $\alpha$ is an irrational number, $\{ \alpha s(n) \mod 1\}_{n \in \mathbb{N}}$ is uniformly distributed.

\end{proposition}

\begin{proof}
This follows from Theorem \ref{uniform} for $a_i$ being this sequence $a_i$ and $b_i = 1$ for each $i$, so $A(n)$ matches the sum of digits. Then for any $\alpha$ irrational and integer $d \neq 0$ the sequence $\gamma_i = \|d \alpha b_i\| = \|d \alpha\|$ is constant and nonzero, so it doesn't converge to zero.

\end{proof}

This applies to most notions of sum digits such as the traditional base $b$, Zeckendoff coding and Narayana base. There has been  is some literature on the distribution of $s(n)$ for these bases, for instance \cite{DUMONT199719} proves a central limit theorem for the distribution of $s(n)$.

\begin{proposition}

Let $w_n$ be the number obtained when writing the number $n$ in base $2$ and reading it in base $3$ \cite{A005836}.  If $\alpha$ is an irrational number, then $\alpha w_n \mod 1$ is uniformly distributed. 

\end{proposition}

\begin{proof}
This is an application of Theorem \ref{uniform} for $a_n=2^{n-1}$ and $b_n = 3^{n-1}$. In order to prove it applies we need to show that for each irrational $\alpha$ and $d \neq 0$, the sequence $\|d \alpha 3^n\|$ doesn't converge to zero. But notice that the only way  $\|d \alpha 3^n\|$ can converge to zero is if the base 3 representation of $d \alpha$ is eventually made up of only 0's or of only 2's, but that is inconsistent to the irrationality of $d \alpha$. So all the assumptions apply.
\end{proof}

This sequence was first explored by Szekeres, as it is the greedy sequence of natural numbers avoiding $3$-term arithmetic progressions. Let us briefly mention two other sequences for which the uniformity of $\alpha a_n \mod 1$ for any irrational $\alpha$ follows similarly from Theorem \ref{uniform}:

\begin{itemize}

\item The Moser-de Bruijn sequence \cite{A000695}, in which $a_n$ is obtained from writing $n$ in base $2$ and reading in base $4$. This sequence is also well-studied for enjoying many additive properties.

\item The sequence of $n$'s such that ${3n \choose n}$ is odd \cite{A003714}. These are exactly the numbers without consecutive $1$'s in their base $2$ representation, so the $n$-th number with this property is obtained by writing $n$ in base Fibonacci and reading it in base $2$.

\end{itemize}

Let us now see an example where the distributions achieved are not uniform, and the signals involved not algebraic.

\begin{proposition}

Let $f_n$ be the sequence of numbers that can be written as a sum of distinct factorials, arranged in increasing order, where $0!$ and $1!$ are not considered distinct \cite{A059590}. 

The distribution of $\frac{1}{e} f_n \mod 1$ converges to a non-uniform, smooth measure, while  $e f_n \mod 1$ does not converge to a distribution at all.

\end{proposition}

\begin{proof}

Notice that $f_n$ is the replacement sequence of $a_n = 2^{n-1}$ by $b_n = n!$. And we also have

$$\|n! e\| \asymp \frac{1}{n} \text{  and  } \|n! \frac{1}{e}\| \asymp \frac{1}{n}$$
(since for instantance $n! \frac{1}{e} = \frac{n!}{0!}-\frac{n!}{1!}+...\pm \frac{n!}{n!} + O(\frac{1}{n})$).

So indeed for $\alpha = e$ or $\frac{1}{e}$, the sequence $\|\alpha n!\|$ converges to $0$, just not quite quickly enough for Theorem \ref{replace} to apply, since their sum diverges. 

So instead of applying it, we will for each $d$ compute the limits of $\hat \mu_n(d)$ ``by hand", where $\mu_n = \frac{1}{n} \sum_{i = 0}^{n-1} \delta(\alpha A(i))$, using the same recurrences from the proof of Theorem \ref{replace}. It is useful to look only at the subsequence $\hat \mu_{2^n}(d)$ - if a limit over this subsequence exists, so does a limit over all $n$ (simply by breaking down $n$ as a sum of powers of $2$). It satisfies the simpler recursion

$$\hat \mu_{2^n}(d) = \frac{(1+e(d\alpha n!))}{2} \hat \mu_{2^{n-1}}(d).$$

Hence the limit we are interested in is just the infinite product:

$$\lim_{n \rightarrow \infty} \hat \mu_{2^n}(d) = \prod_{n = 1}^{\infty} \left(\frac{1+ e(d \alpha n!)}{2}\right).$$

Let us first take on the case $\alpha = e$, where $e(d \alpha n!) = 1+ 2 \pi i \frac{d}{n} + O(\frac{1}{n^2})$. The last term $O(\frac{1}{n^2})$ makes little difference for convergence, so we get something like

$$\prod_{n = 1}^{\infty}\left(1+  \pi i d \frac{1}{n}\right),$$
which is a product that doesn't converge! Even though the absolute value of this complex number converges to $\prod_{n = 1}^{\infty}(1+ \pi^2  d^2 \frac{1}{n^2})$, the argument of the $n$-th term is around $2 \pi  d\frac{1}{n}$ - the sum of these diverges, so the product just swirls forever in a circle. It follows that $\mu_{n}$ can't converge to a distribution, or the sequence of their $d$-th Fourier coefficienst of would also converge for each $d$. This example shows that the decrease condition $\sum_n \|b_n \alpha \| < \infty$ in Theorem \ref{replace} is about as sharp as possible.

As for $\alpha = \frac{1}{e}$, we get instead $e(d\alpha n!) = 2 \pi i d \frac{(-1)^{n+1}}{n} + O(\frac{1}{n^2})$. So we get a product like

$$\prod_{n = 1}^{\infty}\left(1+ \pi i d \frac{(-1)^{n+1}}{n}\right),$$
which converges, thanks to the alternation. This can be used to show that $\hat \mu_{n}(d)$ converges for each $d$, which in turn implies that the sequence $\mu_n$ converges to a certain limit measure. It is precisely the measure whose $d$-th Fourier coefficient is equal to

$$\prod_{n = 1}^{\infty}\left(\frac{1+ e(d n!\frac{1}{e})}{2}\right)$$
for each $d$. Because these Fourier coefficients are non-zero we know the resulting distribution is not uniform. Furthermore, we can show this this distribution is smooth!

The fractionary part of $dn!\frac{1}{e}$ is $\frac{d}{n} + O(\frac{d}{n^2})$, which for $n$ between $1.2d$ and $1.8d$ is between $0.1$ and $0.9$. This implies $|(\frac{1+e(d n!\frac{1}{e})}{2} )| \le 0.96$ for each such $n$. At the same time, each term of the infinite product is at most $1$ in absolute value. This yields the inequality

$$|\hat \mu(d)| \le 0.96^{0.6d} \le 0.98^d$$.

The exponential decay of these Fourier coefficients implies that the limit distribution is smooth. Here is a picture of this distribution:

\begin{figure}[h]
\centering
  \includegraphics[width=250pt]{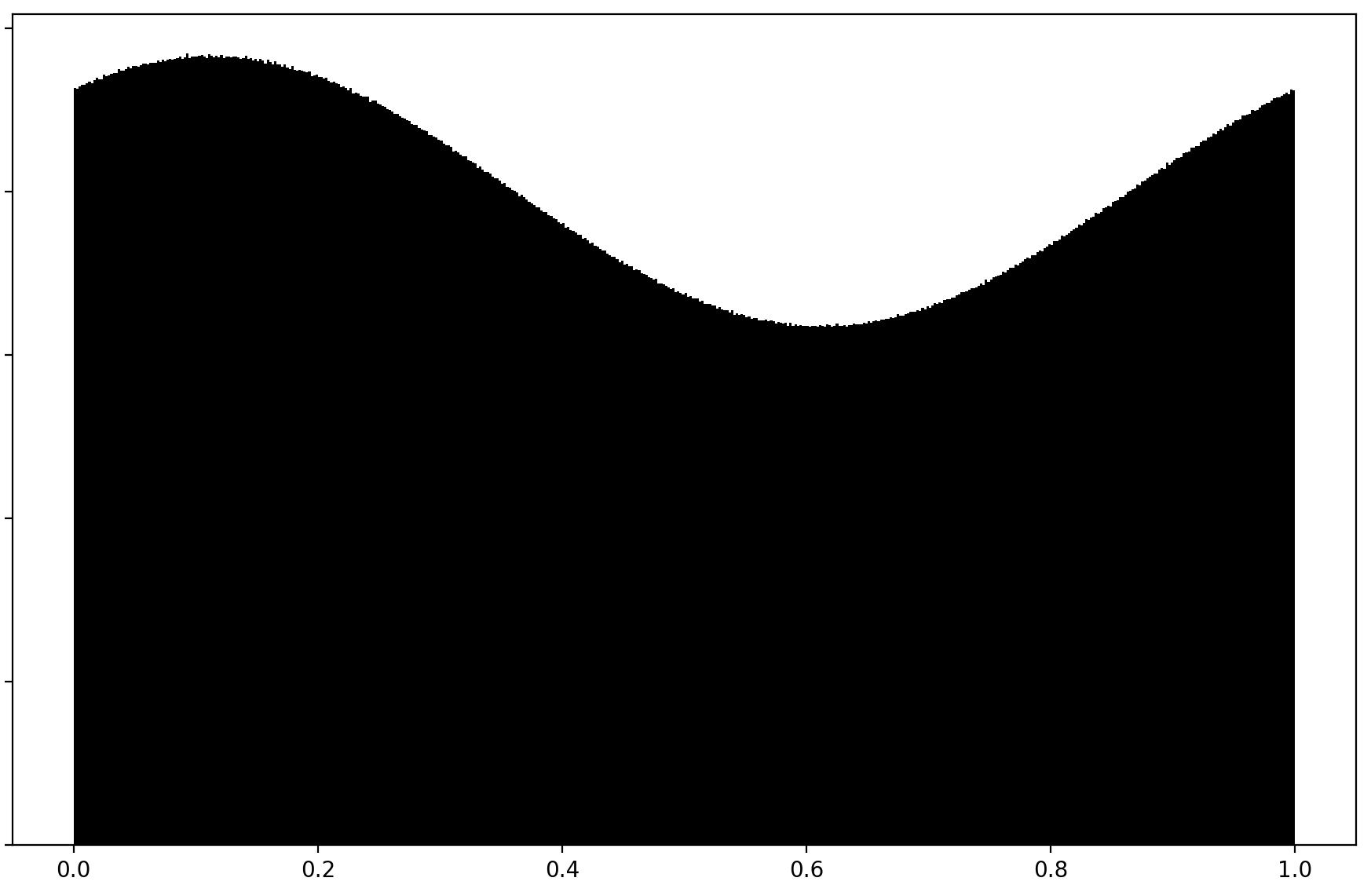}
  \caption{Histogram of $\frac{1}{e} f_n$ for $n$ up to $10^7$}
\end{figure}

\end{proof}

\section{Absolute Continuity}\label{abscontinuity}
\

We showed that for the Hofstadter $H$ sequence the distribution of $\alpha H(n) \mod 1$ converges to a certain measure $\mu$ - in this section we will show this measure is absolutely continuous (with respect to the Lebesgue measure). This means that $\mu(E) = 0$ for any set $E$ of Lebesgue measure zero - or equivalently by the Radon-Nikodym theorem \cite{lang1969analysis} that $\mu$ is of the shape $u dx$ for some non-negative integrable function $u$. So concretely, the absolute continuity of $\mu$ means that Figure \ref{fig:h3fig} of the distribution of $\alpha H(n) \mod 1$ pictures an actual function, rather than say a singular measure. The proof is straightforward.

\begin{lemma} \label{absolute continuity}

Let $a_1,a_2,...$ be a sequence of natural numbers. Assume there are constants $B$ and $C$ such that $a_n \le Bn$ for each $n$, and such that each natural number appears at most $C$ times in the sequence.

If $\alpha$ is irrational and the distribution of $\alpha a_n \mod 1$ converges to a measure, then this measure is absolutely continuous.

\end{lemma}

\begin{proof}

Let $\mu_n = \frac{1}{n} \sum_{i = 0}^{n-1} \delta(\alpha a_i)$, and let $\mu$ be the measure that this sequence converges to. Let $I$ be an arbitrary closed interval of the circle. Let $f$ be a continuous function between the characteristic function of $I$ and the characteristic function of some interval $I'$ of length twice the length of $I$ containing it. We have

$$\mu(I) \le \int f(x) d \mu.$$

At the same time, if $n$ is large enough the number of $a_k$'s with $k\le n$  such that $\alpha a_k \mod 1 \in I'$ is at most $2BCn|I'|$. This is because all these $a_k$'s are less than $Bn$, and provided $n$ is large enough, up to $Bn$ at most $2Bn|I'|$ naturals $m$ satisfy $\alpha m \in I'$ (this follows from the equidistribution of $\alpha m \mod 1$, since $\alpha$ is irrational). Each of these can appear at most $C$ times between the $a_k$'s, hence the inequality.

Since $f$ is $0$ outside $I'$ and at most $1$ inside it, this means that for each large $n$65
$$\int f(x) d \mu_n \le 2BC |I'| = 4BC|I|.$$

But from the weak convergence, $\int f(x) d \mu_n \rightarrow \int f(x) d \mu$ as $n \rightarrow \infty$. Hence we may conclude

$$\mu(I) \le 4BC|I|$$
for any interval $I$, which suffices to prove $\mu(E) = 0$ for any set $E$ of Lebesgue measure zero (since any measure zero set can be approximated above by a countable set of intervals with arbitrarily small total lenghts). Hence $\mu$ is absolutely continuous with respect to the Lebesgue measure, as desired.

\end{proof}

\begin{remark}

We can go a bit further - while the Radon-Nikodym theorem implies that $\mu$ is of the shape $u dx$ for some non-negative function $u \in L^{1}$, the bound 

$$\mu(I) \le 4BC |I|$$
for any interval $I$ implies that the density function $u$ actually belongs to $L^{\infty} \subset L^{1}$. That's because if the set $\{x: u(x) > 4BC \}$ has positive measure, by approximating this set with unions of intervals one can argue that one of those intervals has to satisfy $\mu(I) > 4BC |I|$, which would contradict the inequality. So $u(x) \le 4BC$ a.e., which implies $u \in L^{\infty}$. Concretely, this means the function pictured in Figure \ref{fig:h3fig} is bounded.

\end{remark}

\begin{corollary}

The limit measure of the numbers $\alpha H(n) \mod 1$ is absolutely continuous. This still holds if $\alpha$ is replaced by a number in $\mathbb{Z}[\alpha]$.

\end{corollary}

It suffices to check the assumptions of Theorem \ref{absolute continuity} apply: indeed because each shift $h_{i-1}$ is equal to $\alpha h_i$ plus an exponentially small error, it follows that $H(n) = \alpha n + O(1)$. Therefore $H(n)$ has at most linear growth, and each natural numbers appears at most a constant number of times in this sequence.


Going deeper in the question of how smooth is the limit distribution $\mu$ of $\alpha H(n) \mod 1$ (Figure \ref{fig:h3fig}), we ask:

\begin{question}
What is the rate of decrease of the Fourier coefficients $\hat \mu(n)$ of this distribution?

\end{question}

We have been unable to answer this question - the most straightforward approach seems to require a detailed understanding of the numbers $d \alpha h_n \mod 1$ for $n \ll \log d$, which does not look easy. Certainly one has that $\sum | \hat \mu(n)|^2$ converges, since the limit measure $\mu$ is inside $L^{\infty} \subset L^2$. Empirically it seems $|\hat \mu(n)|$ is on average decreases like $|n|^{-1.47}$. If  that is true, it would suffice to establish the following:

\begin{conjecture}

The probability density function of $\mu$ is continuous but not differentiable.

\end{conjecture}

This is because the convergence of  $\sum |\hat \mu(n)|$ guarantees continuity, while the divergence of $\sum |n \hat \mu(n)|^2$ guarantees non-differentiability.
\\

We also remark on the relationship of Lemma \ref{absolute continuity} with the Ulam sequence.

\begin{remark}

Conditional on the hypotheses that the Ulam sequence grows at most linearly and that Steinerberger's signal $\alpha$ is irrational, if the limit distribution $\alpha u_n \mod 1$ exists them from this lemma it would follow it is absolutely continuous.
\end{remark}

The problem of determining the natural density of this sequence was considered by Ulam. Muller predicted it to be zero, but more extensive recent computations are suggestive that it is positive \cite{steinerberger2015hidden}. If so we could apply Theorem \ref{absolute continuity} with $C = 1$ since the sequence is increasing by definition. Before Steinerberger's discovery the observed positive density seemed inconsistent with heuristic analysis, but taking the hidden distribution into account Gibbs \cite{gibbsheuristic} manages to heuristically explain the linear growth. With that in mind, it seems likely to us that if these hypotheses about the Ulam sequence are ever proved they will be proved at the same time - though they seem currently out of reach.

\newpage
\section{A closer look at Figure 2}\label{closerlook}
\

In this section we will look at some strange physical features of this distribution. We print its picture again for convenience:

\begin{figure}[h]\label{h3fig2}
\centering
  \includegraphics[width=300pt]{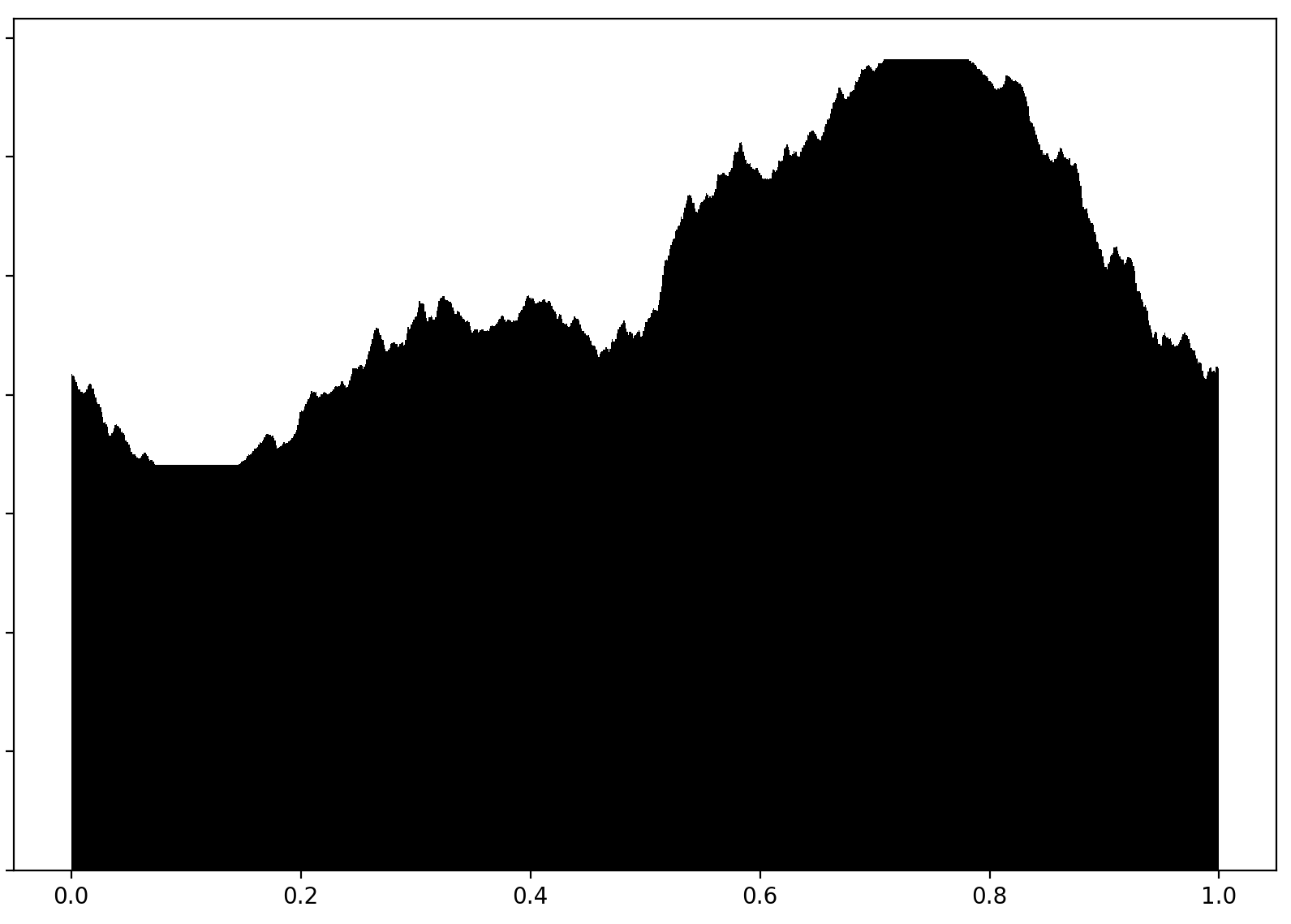}
  \caption{Copy of Figure 2, a histogram of $\alpha H(n) \mod 1$ for $n$ up to $10^7$}
\end{figure}

Notice how there is a ``valley" and a ``hill", and each has an interval where they look flat. Nothing is lower than the valley or higher than the hill, and the height of the valley is exactly half the height of hill. There also seems to be a symmetry between the valley and the hill as if they are figure and ground to each other - here is a zoomed in picture of how they fit:

\begin{figure}[h]
\centering
  \includegraphics[width=250pt]{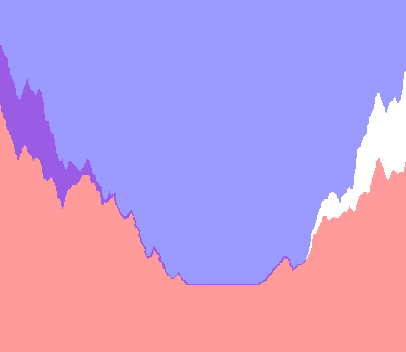}
  \caption{The valley is in red and a vertically reflected copy of the hill in blue}
\end{figure}

They fit perfectly, even beyond their flat interval, althought they eventually drift apart. Let us sketch explanations for all these observations (though we welcome the reader to figure out these puzzles on their own).

Because $H(n)$ is the right shift of $n$ in the Narayana base, for each $m$ there is either exactly one or exactly two $n$'s with $H(n) = m$. There is always at least one solution - the left shift of $m$ - and the $m$'s that have two solutions are precisely the ones of the shape $m=H(n)$ if $n$ is a number that ends with a $1$ in its Narayana base representations - in which case the two solutions will be this $n$ and the left shift of $m$. This relates to why the height of the hill is exactly twice that of the valley: eventually we will see that each $m$ such that $m \alpha \mod 1$ is around the valley has exactly one solution, and each $m$ such that $m \alpha \mod 1$ is around the hill has exactly two.

Let us then write the sequence $H(n)$ as a union of the sequence of all natural numbers with the sequence of the numbers of the shape $m=H(n)$ where $n$ ends with a $1$ in its Narayana base representation, which in turn breaks down the the distribution of $\alpha H(n) \mod 1$ as a sum of two distributions. Because $\alpha$ is irrational, the sequence $\alpha n \mod 1$ where $n$ goes over all natural numbers is equidistributed, so the former distribution is uniform. This uniform component gives a flat baseline for the overall distribution - now all we have to do to explain the flatness of the valley is to prove the other component is zero in a certain interval.

But for any $n$ that ends with a $1$ in the Narayana base it can't have the next two digits $2$ and $3$. And remember $\|\alpha h_i\|$ decreases rapidly, so if $n$ is a sum of $h_i$'s, $\alpha n$ will be the sum of the associated $\alpha h_i$'s, which will generally be small. Based on that one can give a bound to where $\alpha H(n) \mod 1$ of such $n$ can be (we have to be a bit careful and account for the ``signs" of the first few $\alpha h_i \mod 1$, getting a different upper and lower bound), which suffices to prove the existence of an interval in which $\alpha H(n) \mod 1$ of such $n$ can never arrive, which corresponds to the valley. The exact endpoints of this interval however seem hard to compute, since they involve an infinite sum that takes into account the signs of $\theta^n+\bar{\theta}^n$, where $\theta$ is a non-real root of $x^3= x^2+1$.

Let us now explain the fit between the valley and the hill. We showed how the distribution $\mu$ can be seen as a sum of an uniform distribution and the distribution only over the $n$'s that appear twice on the $H$ sequence - call that distribution $\eta_1$. But conversely, we can see $\mu$ as two times an uniform distribution minus the numbers that appear only once in the $H$ sequence - call the distribution of those numbers $\eta_2$ (this subtraction corresponds to the vertical reflection of the hill). Now proving that the valley and the hill fit into each other is the same as proving that there is a shift of $\eta_1$ that is equal to $\eta_2$ in a certain interval.

But one can show that $m$ appears in the sequence associated to $\eta_2$ if and only if $m+2$ appears in the sequence associated to $\eta_1$, except for a small proportion of exceptions. This will show that the distribution $\eta_1$ is roughly equal to the shift of $\eta_2$ by $2 \alpha$ - indeed one may confirm in the picture that the horizontal distance between the valley and the hill is equal to $2 \alpha \mod 1$. If $I$ is the interval of the hill in $\mu_2$, it suffices to prove that any exceptional $m$, the number $\alpha m \mod 1$ lies outside $I$. 

Playing a bit with the Narayana base it is easy to find a bunch of restrictions on such exceptional $m$. They never have the digits $2$ or $3$ in the Narayana base, and they have the property that the digit $1$ appears if and only if the digit $4$ appears. Again, because of the fast decrease of $\| \alpha h_n\|$, it is easy to use these first conditions on the first few digits to show that $\alpha m \mod 1$ avoids a certain interval for any exceptional $m$. 

Because this condition takes into account a few more digits than before, it is not hard to believe it makes sure that $\alpha m \mod 1$ is always outside a slightly larger interval than I (hence the fit going a bit beyond $I$), which completes our explanations.

Some mysteries about this picture remain, as we still don't know for sure if the function we are looking at is continuous.


\section{Further $H$ sequences, and the more general problem}\label{htype}
\

Let us begin this section overviewing the signals in generalized $H$ sequences for $d \le 5$.

For $d = 1$ the sequence $F(n)$ is defined by $F(1) = 1$ and then the recursion

$$F(n) = n - F(n-1).$$

It has the closed formula $F(n) = \lceil \frac{n}{2} \rceil$, which can also be seen as a right shift of $n$ in base $2$ (with $2^0$ is shifted to $1$). So for instance from Theorem \ref{uniform} it follows easily that $\alpha F(n) \mod 1$ is uniform for any irrational $\alpha$.

The $d=2$ sequence is defined by $G(1) = 1$ and

$$G(n) = n - G(G(n-1)).$$

The number $G(n)$ can be proven to be the right shift of $n$ in the Fibonacci base (a famous procedure for converting kilometers to miles!). But this case is still exceptional for having the simple closed formula $G(n) = \lceil \frac{n}{\phi} \rceil$ (no such formula holds for the $H$ sequence we previously overviewed). This allows us to understand exactly the signal it has at $\alpha = \frac{1}{\phi}$ (the root of $x^2+x=1$), and the associated distribution of $\frac{1}{\phi}  G(n)\mod 1$:

\begin{figure}[h]
\centering
  \includegraphics[width=250pt]{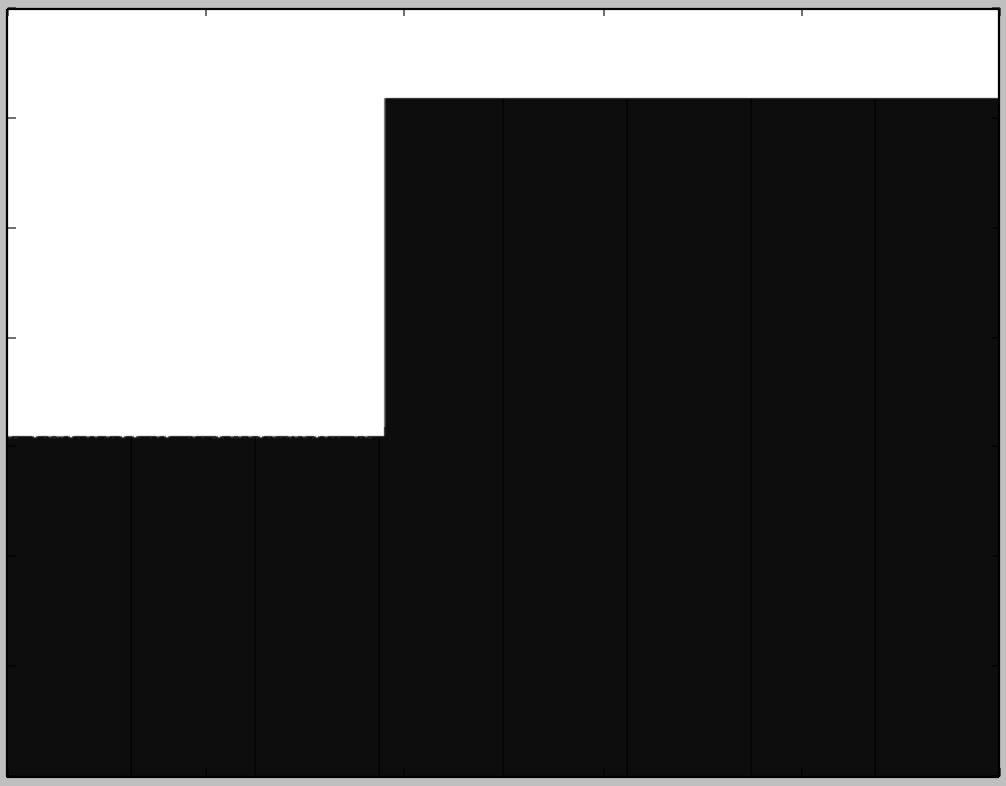}
  \caption{Histogram of $\frac{1}{\phi} G(n) \mod 1$ for $n$ up to $10^7$}
\end{figure}

For a quick proof that this distribution is a step function, notice that $\frac{1}{\phi} G(n)= (\phi-1)G(n)$, so its fractionary part is

$$\{\frac{1}{\phi} G(n)\} = \{\phi G(n) \} = \{ n - \phi(-\{\frac{n}{\phi} \} +1)\} = \{-\phi (1-\{\frac{n}{\phi}\}) \}.$$

But $\frac{1}{\phi}$ is irrational, so $(1-\{\frac{n}{\phi}\})$ is equidistributed in $[0,1]$, so $-\phi (1-\{\frac{n}{\phi}\})$ is equidistributed in $[-\phi, 0]$. The fractionary part of that distribution is exactly the step function we see in the picture. With the formula $G(n) = \lceil \frac{n}{\phi} \rceil$ it is also simple to prove that the distribution of $\beta G(n) \mod 1$ is not uniform are precisely when $\beta \in \mathbb{Q}[\phi]$, and to describe those limit distributions.

Now the sequence with $4$ iterations:

$$I(n) = n - I(I(I(I(n-1)))).$$

It has a signal at $\alpha$ the positive root of $x^4+x=1$, with the following picture for the distribution of $\alpha I(n) \mod 1$:

\begin{figure}[h]
\centering
  \includegraphics[width=250pt]{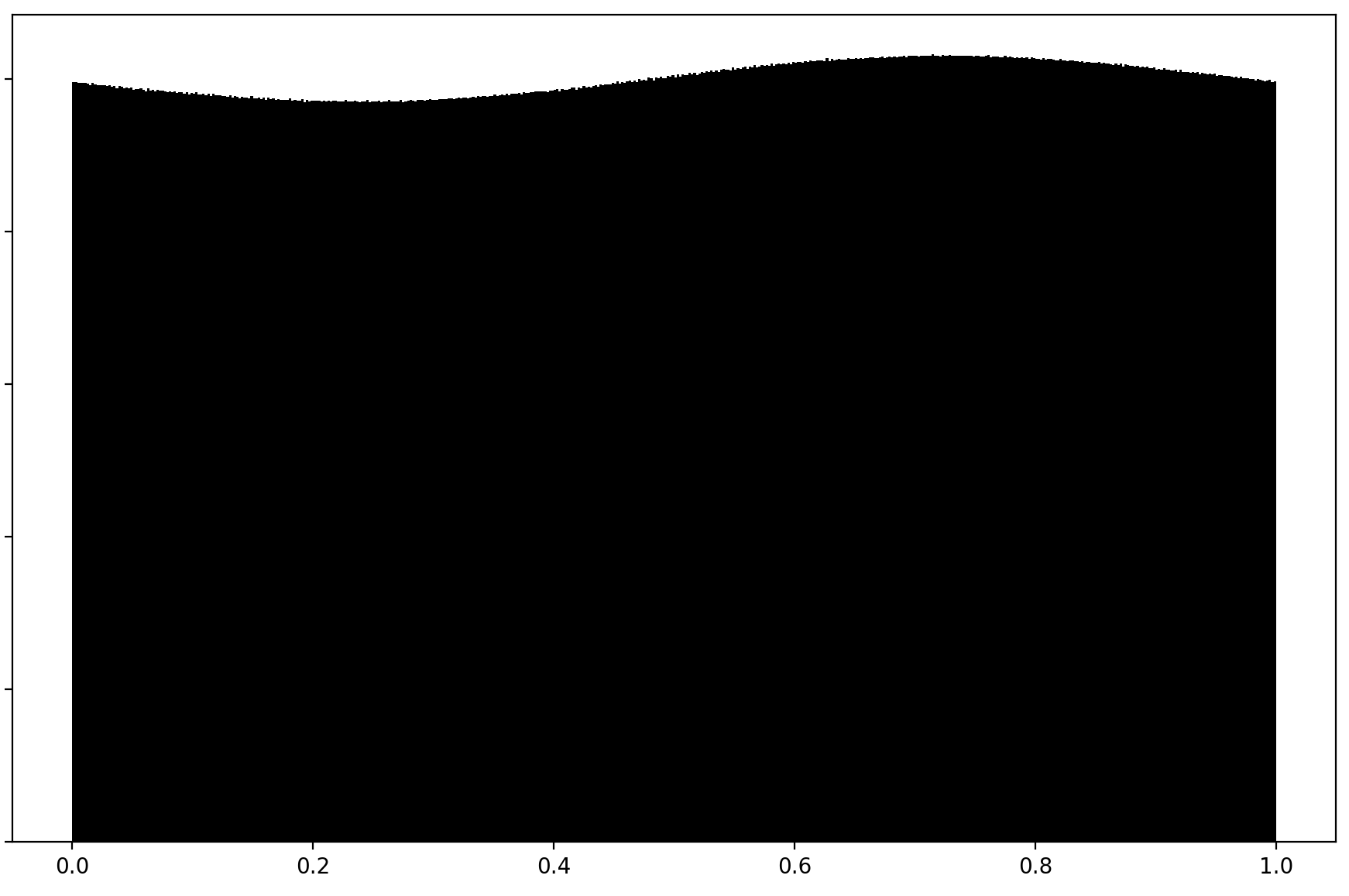}
  \caption{Histogram of $\alpha I(n) \mod 1$ for $n$ up to $10^7$}
\end{figure}

The existence of this limit distribution for the $I$ sequence follows from Theorem \ref{replace}. The number $I(n)$ is equal to the right shift of $n$ when written in base $i_n$, where $i_n$ is a sequence satisfying the linear recurrence $i_{n} = i_{n-1}+i_{n-4}$. This sequence is of finite signature, and it is easy to show that the root $\alpha$ of $x^4+x=1$ has the property that $\|\alpha i_n \| \rightarrow 0$ quickly, and that this is false for any $\beta$ outside $\mathbb{Q}[\alpha]$, so Theorems \ref{replace} and \ref{uniform} apply. The same methods as the $H$ sequence can also prove that this distribution is not uniform and is absolutely continuous. It does look like the Fourier coefficients of this distribution decrease faster than in the $H$ case, making it more smooth - it seems to be at least differentiable, but only finitely many times.

The next case is the $J$ sequence, with 5 iterations has the observed distribution at the signal $\alpha$ root of $x^5+x=1$ (which is also root of $x^3-x-1=0$):

$$J(n) = n - J(J(J(J(J(n-1)))))$$

\begin{figure}[h]
\centering
  \includegraphics[width=250pt]{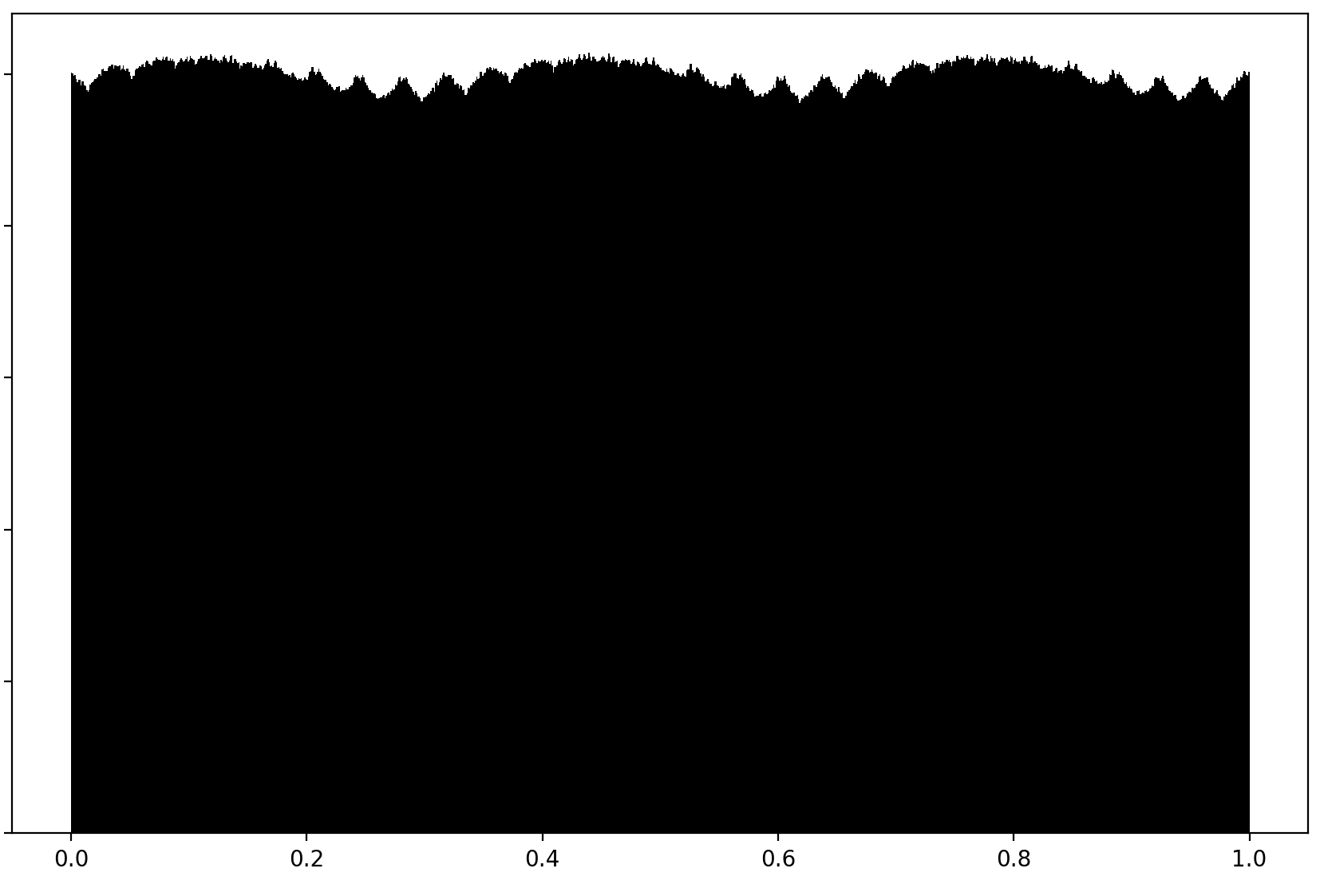}
  \caption{Histogram of $\alpha J(n) \mod 1$ for $n$ up to $10^7$}
\end{figure}

Again $J(n)$ is the right shift in the base of a recurrent sequence $j_n$ which is of finite signature and has the property $\|\alpha j_n \| \rightarrow 0$, so \ref{replace} and \ref{uniform} apply.

The cases $d \ge 6$ however have a different behavior - it is still the case that these sequences have a closed formula as the right shift of a linear recurrence $a_n$, but this recurrence no longer has any irrational numbers with the property $\|\alpha a_n \| \rightarrow 0$. The main culprit is that for $d \ge 6$ the equation $x^d+x =1$ has more than $1$ root of absolute value greater than $1$.

This is related to the concept of a Pisot-Vijayaraghavan number - which is defined to be an algebraic real number $\alpha$ with absolute value more than $1$ whose all other algebraic conjugates have absolute value less than $1$. These numbers have been extensively studied \cite{vijayaraghavan_1941, Pisot1938LaRM, salem1944} for instance Pisot \cite{Pisot1938LaRM} proved they are precisely the ones with the property that $\| \alpha^n \| \rightarrow 0$ quickly. Another incredible property of these numbers conjectured by Vijayaraghavan and proved by Salem \cite{salem1944}, is that the limit of a sequence of Pisot-Vijayaraghavan numbers is still Pisot-Vijayaraghavan.

 If $\alpha$ a Pisot-Vijayaraghavan number, and $a_n$ a linear recurrence of integers with the same characteristic polynomial as $\alpha$, one can easily show that $\| \alpha a_n \| \rightarrow 0$ quickly. The converse however is a bit more subtle --- there are numbers with this property that are not Pisot-Vijayaraghavan. The following theorem lays out a more precise description of the real numbers $\alpha$ and linear recurences $a_n$ with $\| \alpha a_n \| \rightarrow 0$:

\begin{theorem}\label{alpharecurence}

Let $P$ be an irreducible monic polynomial in $\mathbb{Z}[x]$ and let $K \hookrightarrow \mathbb{C}$ be the splitting field of $P$. We will say a field automorphism $\sigma \in Gal(K/\mathbb{Q})$ is ``big" if there exist (not necessarily distinct) roots $\alpha_i, \alpha_j$ of $P$ satisfying $|\alpha_i|, |\alpha_j| \ge1$ such that $\sigma (\alpha_i) = \alpha_j$. Let $a_n$ be a non-zero linear recurrent sequence of integers with characteristic polynomial $P$.

If $\alpha$ is a real number such that

$$\|\alpha a_n\| \rightarrow 0,$$
then $\alpha \in K$ and $\sigma (\alpha) = \alpha$ for each big automorphism $\sigma$ of $K$. 

And as an approximate converse, if $\alpha\in K$ is fixed by each big automorphism of $K$, then there is a non-zero integer $d$ such that $\|d\alpha a_n\| \rightarrow 0$ quickly.

\end{theorem}

So Theorem \ref{alpharecurence} states that the set of of real numbers $\alpha$ of interest is roughly the fixed field of the subgroup of $Gal(K/\mathbb{Q})$ generated by big automorphisms. An important particular case is when $P$ has a single root satisfying $|\alpha| \ge 1$ (that is, when $\alpha$ is a Pisot-Vijayaraghavan number). In this case the big automorphisms are precisely those that fix $\alpha$, so the subfield of $K$ fixed by those automorphisms is precisely $\mathbb{Q}[\alpha]$, which is indeed the set of interest in the cases $d \le 5$ of the $H$ sequence, all of which are atatched to Pisot-Vijayaraghavan numbers.

\begin{proof}

Let $b_n$ be the closest integer to $\alpha a_n$. Because $b_n$ is close to $\alpha a_n$, the $b_n$ ``nearly" satisfies the same linear recurrence as $a_n$. But the $b_n$ are integers, and the linear recurrence has integer coefficients, so for large $n$ the recurrence is satisfied exactly. So $b_n$ is also linear recurrent with characteristic polynomial $P$. Now let

$$a_n = \sum_i c_i \alpha_i^n \text{ and } b_n = \sum_i c_i' \alpha_i^n$$
be the closed formulas for $a_n$ and $b_n$ respectively, where $\alpha_i$ are the roots of $P$. By applying the linear operator of characteristic polynomial $\frac{P(x)}{(x-\alpha_i)}$ to these sequences it follows each $c_i, c_i' \in \mathbb{Q}[\alpha_i] \subset K$. And taking the limit of the ratio between these formulas proves that $\alpha \in K$.

Now because the $a_n$ are integers we have $\sigma(a_n) = a_n$, therefore

$$\sum \sigma (c_i) \sigma (\alpha_i)^n = \sum c_i \alpha_i^n.$$

From the uniqueness of this closed formula, we get that if $\sigma (\alpha_i) = \alpha_j$, then $\sigma (c_i) = c_j$. That is, the rationality of $a_n$ requires a certain consistency is required between the $c_i$. Similarly, $\sigma (c'_i) = c'_j$. 

But we know that $b_n - \alpha a_n \rightarrow 0$, which implies that when we write $b_n - \alpha a_n$ as a combination of powers of $\alpha_i$, the coefficient associated to each root $|\alpha_i|\ge 1$ is zero. That is, $c_i' = \alpha c_i$ for any root $|\alpha_i|\ge1$.

Now for a given big automorphism $\sigma$, let $\sigma(\alpha_i) = \alpha_j$ where $|\alpha_i|, |\alpha_j| \ge 1$ we get:

$$\alpha c_j = c_j' = \sigma (c_i') = \sigma (\alpha c_i) = \sigma (\alpha) \sigma (c_i) = \sigma (\alpha) c_j$$

And we know that $c_j \neq 0$. This is because $P$ is irreducible, so the action of $Gal(K/\mathbb{Q})$ on the the roots $\alpha_i$ is transitive. Because of the consistency between the $c_i$'s this means that if one of them were $0$, all of them would be $0$, which would contradict that the sequence $a_n$ is non-zero. Canceling out $c_j$ yields

$$\alpha = \sigma (\alpha).$$

This holds for any big automorphism of $K$, as desired.
\\

For the converse, let $\alpha \in K$ be fixed by each big automorphism of $K$. The first thing we notice is that complex conjugation is a big automorphism, so $\alpha$ must be real - it is therefore meaningful to talk about $\|\alpha a_n\|$. Let $\alpha_1$ be a root with absolute value at least $1$ of $P$. From the consistency between the coefficients $c_i$ of $a_n$ and the transitivity of the action of $ Gal(K/\mathbb{Q})$, we can write

$$a_n = \sum_{\sigma}  \sigma(c_1) \sigma (\alpha_1)^n,$$
where the sum is over every automorphism of $K/\mathbb{Q}$. And define

$$b_n = d\sum_{\sigma} \sigma(\alpha) \sigma(c_1) \sigma(\alpha_1)^n$$
where $d$ is an integer to be chosen soon. Its coefficients satisfy the same consistency check as $a_n$, so $\sigma b_n = b_n$ for each  $\sigma\in Gal(K/\mathbb{Q})$, which implies the $b_n$ are rationals. Further, if we choose $d$ so that $d\alpha c_1$ is an algebraic integer (for instance the leading coefficient of the minimal polynomial of $\alpha c_1$ works), then $b_n$ is a combination of algebraic integers, so it is not only rational but an integer.

Now we have the equation

$$d\alpha a_n - b_n =  d\sum_{\sigma} (\alpha -\sigma(\alpha))  \sigma(\alpha_1)^n$$
and we also know that each $\sigma$ such that $|\sigma(\alpha_1)|\ge 1$ is a big automorphism (since it takes $\alpha_1$ to another large root), so for each such $\sigma$ the associated coefficient $(\alpha -\sigma(\alpha))$ vanishes. It follows that

$$d\alpha a_n - b_n \rightarrow 0$$
which means $\|d\alpha a_n\| \rightarrow 0$ quickly, as desired.

\end{proof}

The following statement gives a setting in which integrality is clean, revealing some additional algebraic structure:

\begin{theorem}\label{integrality}

Let $P$ be a monic polynomial in $\mathbb{Z}[x]$. Consider the set $R$ of real numbers $\beta$ such that $\|\beta a_n\| \rightarrow 0$ for every sequence of integers $a_n$ satisfying the linear recurrence of characteristic polynomial $P$. Then:

\begin{itemize}

\item $R$ is a ring.

\item If $P(0) = \pm 1$, the ring $R$ is integral over $\mathbb{Z}$ (that is, each of its elements is an algebraic integer).

\end{itemize}

Remark: The condition that $\|\beta a_n\| \rightarrow 0$ for every sequence satisfying a certain recurrence rather than a particular one may seem much stronger, but in many cases it is not. Indeed, $R$ is simply the set of numbers satisfying the condition for the sequence with initial terms $(0,...,0,1)$, since any other sequence satisfying the recurrence is an integral combination of shifts of this sequence.

\end{theorem}

\begin{proof}

Closure under addition is clear. For multiplication, assume that$ \alpha, \beta \in R$. Then if $a_n$ is any sequence satisfying the recurrence $P$, because $\beta \in R$, the sequence $b_n$ of the closest integer to $\beta a_n$ eventually satisfies the recurrence too (repeating previous arguments we've made). And because $\alpha \in R$, it follows that $\| \alpha b_n \| \rightarrow 0$. That is $\| \alpha (\beta a_n \pm \|\beta a_n\|) \| \rightarrow 0$, which implies $\|\alpha \beta a_n \| \rightarrow 0$. That is for any sequence satisfying the linear recurrence $P$, so $\alpha \beta \in R$, as desired.

For the integrality, the key observation is that when $P(0) = \pm 1$ we can give the set of integer sequences satisfying $P$ the structure of an $R$-module. Let $\alpha \in R$ act on a sequence $a_n$ by taking it to the sequence of closest integers to $\alpha a_n$. We don't get a full sequence right away - this new sequence only satisfies the recurrence for large $n$. But if $P(0)=\pm 1$ one can reconstruct the initial terms of the sequence using the recursion backwards. So we can make $\alpha$ act taking a sequence satisfying the recurrence to another such sequence. Hence the $R$-module structure. Notice that this $R$-module is finitely generated over $\mathbb{Z}$, since we may use the shifts of the sequence with initial terms $(0,...,0,1)$ as a basis. And every ring $R$ with an $R$-module that is finitely generated over $\mathbb{Z}$ is integral, concluding the proof.

The condition $P(0) =\pm 1$ is required; for instance for $P(x) = x-2$ the ring of $\alpha$ such that $\|\alpha 2^n\| \rightarrow 0$ is $\mathbb{Z}[\frac{1}{2}]$, which is not made up of algebraic integers. In fact it is simple to establish integrality over $\mathbb{Z}[\frac{1}{P(0)}]$ generally.

\end{proof}

\begin{example}

Let $a_n = 0,0,0,0,0,1$ for $1\le n \le 6$, and let it satisfy the recurrence $a_n = 3a_{n-1}+6a_{n-2}-4a_{n-3}-5a_{n-4}+a_{n-5}+a_{n-6}$. Then $\| \beta a_n \| \rightarrow 0$ precisely for $\beta \in \mathbb{Z}[\frac{1+\sqrt{13}}{2}]$.

\end{example}

\begin{proof}

The characteristic polynomial of this recurrence is the minimal polynomial of $\frac{1}{2 \cos(\frac{2 \pi}{13})} = \frac{1}{\zeta_{13}+\zeta_{13}^{-1}}$, with conjugates $\frac{1}{2 \cos(k\frac{2 \pi}{13})} = \frac{1}{\zeta_{13}^k+\zeta_{13}^{-k}}$ for $k =1,2,3,4,5,6$. Only two of these are bigger than $1$ in absolute value: $k = 3$ and $k =4$.

The splitting field is the maximal real subfield of $\mathbb{Q}[\zeta_{13}]$. Its Galois group is cyclic and each automorphism can be described by an element $a \in (\mathbb{Z}/(13))^*$, acting as $\sigma(\frac{1}{\zeta_{13}^k+\zeta_{13}^{-k}}) = \frac{1}{\zeta_{13}^{ak}+\zeta_{13}^{-ak}}$ (the elements $a$ and $-a$ describe the same automorphism).

The big automorphisms will be the ones either fixing the big roots $k = 3$ or $4$ or taking one to another. Because $3$ and $4$ are both squares mod $13$, the subgroup generated by big automorphisms are only quadratic residues $\mod 13$. These form a subgroup of index $2$ of our Galois group, whose fixed field is $\mathbb{Q}[\sqrt{13}$] (the only quadratic field in $\mathbb{Q}[\zeta_{13}]$).

From Theorem \ref{integrality} our ring of interest is integral is hence a subring of $\mathbb{Z}[\frac{1+\sqrt{13}}{2}]$. One can verify $\beta = \frac{1+\sqrt{13}}{2}$ indeed has the property $\|\beta a_n \| \rightarrow 0$, by checking that the sequence $b_n$ constructed in theorem $\ref{alpharecurence}$ meant to be the sequence of closest integers to $\beta a_n$ is in fact made up of integers for $d=1$, which completes the proof.

\end{proof}

\begin{example}

Let $a_n = 1,2,3,4$ for $n = 1,2,3,4$ and $a_n = 10 a_{n-2}-a_{n-4}$ for $n>4$. Then $\| \beta a_n \| \rightarrow 0$ precisely for $\beta \in \mathbb{Z}[\sqrt{6}]$.

\end{example}

\begin{proof}
This sequence has characteristic polynomial $P(x) = x^4-10x^2+1$, which is the minimal polynomial of $\sqrt{2}+\sqrt{3}$. Its roots are $\pm \sqrt{2} \pm \sqrt{3}$, and the Galois group of its splitting field $K = \mathbb{Q}[\sqrt{2},\sqrt{3}]/\mathbb{Q}$ is $\{1, -1\}^2$, where $\sigma = (\xi_1, \xi_2) \in  \{1, -1\}^2$ acts on $K$ as $\sigma(\sqrt{2}) = \xi_1 \sqrt{2}$ and $\sigma(\sqrt{3}) = \xi_2 \sqrt{3}$. Then the the big automorphisms in this context are  $(1,1)$ and $(-1,-1)$, which generate a subgroup of index $2$ of the Galois group, whose fixed field is $\mathbb{Q}[\sqrt{6}]$. So the ring of interest is a subring of $\mathbb{Q}[\sqrt{6}]$. From here it is simple to check $\| \sqrt{6}a_n\| \rightarrow 0$ so the ring of interest contains $\mathbb{Z}[\sqrt{6}]$, which largest integral subring of $\mathbb{Q}[\sqrt{6}]$ so it must be our ring.

\end{proof}

Finally, let us apply this to the linear recurrences associated with the generalized $H$ sequences.

\begin{theorem}

Let $d\ge 6$ be an integer and $H$ be the $d$-th hofstadter sequence. Then $\alpha H(n) \mod 1$ is uniform for any irrational $\alpha$.

\end{theorem}

We may define a generalized Narayana sequence by $a_i = i$ for $1 \le i \le d$, and

$$a_i = a_{i-1}+a_{i-d}$$
for $i > d$. Then it is simple to prove inductively that $H(n)$ is the right shift of of $n$ in base $a_i$. This equation gives the signature of the sequence $a_i$, hence it is of mild signature and Theorem \ref{uniform} applies. So in order to prove $\alpha H(n) \mod 1$ is uniform for any irrational $\alpha$, it suffices to prove that $\| a_n \alpha \|$ does not converge to zero for any irrational $\alpha$. In order to get that as an application of  Theorem \ref{alpharecurence} we will need two inputs: one about the location of the roots of the characteristic polynomial $x^d-x^{d-1}-1$ of the sequence $a_i$, and another about the Galois structure of this polynomial. First the location of the roots:

\begin{proposition}

For $d \ge 6$ the polynomial $x^{d}-x^{d-1}-1$ has at least two roots strictly outside the unit disk. Indeed, for large $d$ approximately $\frac{d}{3}$ of its roots lie outside the unit disk.

\end{proposition}

\begin{proofsketch} A classic result of Erd\'os and Turán \cite{erdosturan}, \cite{soundararajan2018equidistribution} states roughly that the roots of a monic polynomial with small coefficients all have absolute value close to $1$, and that they are approximately equidistributed around the unit circle. Apply it to our polynomial $x^d - x^{d-1}-1$. And because of the equation

$$1=|x|^{d-1}|x-1|,$$
a root of $1+x^{d-1}=x^{d}$ has absolute value greater than $1$ exactly when $|x-1|<1$. But for $x$ close to the unit circle, this happens when $x$ belongs to the right third of the unit circle (notice that $x = e^{\pm \frac{\pi i}{3}}$ satisfy $|x-1| = 1$). From equidistribution, around $\frac{d}{3}$ of the roots lie there, as desired. 

This proposition is also related to Lehmer's conjecture. This conjecture states that if $P\in \mathbb{Z}[x]$ is not a product of cyclotomic polynomials, then the product of the absolute values of roots outside the unit disk is at least $c$, where $c>1$ is independent of $P$. Our polynomial is not cyclotomic (since it has a real root greater than 1), and each of its roots satisfy $|\alpha| \le 2^{\frac{1}{d}} \rightarrow 1$, so Lehmer's conjecture would imply that our polynomial has many roots of absolute value greater than $1$ for large $d$. And in fact Smyth \cite{smythfact} proved that Lehmer's conjecture holds for any non reciprocal polynomial for $c = \theta \sim 1.324$ the real root of $x^3=x+1$. This applies to our polynomial, hence it suffices to give a complete alternative proof of the proposition, as for $d \ge 6$ the positive real root of $x^d - x^{d-1}-1$ is smaller than $\theta$, therefore at least one other root of absolute value larger than $1$ must exist. A funny observation is that this $\theta$ is exactly the real root of $x^5-x^4-1$, so by very little the polynomial for $d=5$ is not forced to have another root greater than 1, which would erase the hidden signal of the $5$-th Hofstadter sequence.

\end{proofsketch}

And now the Galois group of this polynomial.

\begin{proposition}
\begin{itemize}

\item For $d \neq 5 \mod 6$, the polynomial $x^{d}-x^{d-1}-1$ is irreducible and its Galois group is isomorphic to the full symmetric group $S_d$. 

\item For $d = 5 \mod 6$, the polynomial $x^{d}-x^{d-1}-1$ splits as a product of $x^2-x+1$ and an irreducible factor of degree $d-2$. The Galois group of the irreducible factor is isomorphic to the full symmetric group $S_{d-2}$.

\end{itemize}

\end{proposition}

\begin{proof}

For a proof of irreducibility, we refer to a paper of Selmer \cite{Selmer_1956} on the irreducibility of trinomials. When $d \neq 5 \mod 6$, our  polynomials are equal to the reciprocal of $f(-x)$, where $f$ is one of the polynomials in his Thereom 1 (which one depends on the parity of $d$), which Selmer proves are irreducible if $d \neq 2 \mod 3$. When $d = 5 \mod 6$, our polynomial is the reciprocal of $f(-x)$ where $x$ is the second polynomial of this theorem, which he proves is a product of $(x^2+x+1)$ and an irreducible factor whenever $d \equiv 2 \mod 3$. This completes the proof of the ``irreducibility" part of our statement.

For the the Galois group computations, we refer to a paper of Osada \cite{osada1987} on the Galois groups of trinomials. Whenever $x^d-x^{d-1}-1$ is irreducible, it satisfies the assumptions of his Theorem 1, which implies its Galois group is $S_d$. 

Osada's paper doesn't study the case in which the trinomial is reducible, but his results can be adapted to our needs. Together, his Lemmas 1 to 6 prove that if an irreducible polynomial $f$ of degree $n$ has the property that for each prime $p$ for which $f$ has a double root $\mod p$ this root is the only double root and also has multiplicity exactly $2$, then the Galois group of $f$ is isomorphic to $S_n$. So it suffices to prove the irreducible factor $f$ of $g(x) = x^d-x^{d-1}-1$ has this property.

Indeed, let $\theta$ be a double root of $f$. Then it is a double root of $g(x) = x^d-x^{d-1}-1$, so it is a root of $g'(x) = d x^{d-1} - (d-1)x^{d-2}$, hence the root must be $\theta = \frac{d-1}{d} \mod p$ (since $0$ is not a root of the original polynomial). It already follows that $f$ may have no other double roots mod p. And because this not a root $\mod p$ of the second derivative of $g$, it follows $\theta$ has multiplicity exactly two in $g$. This means $\theta$ it is a root of multiplicity exactly $2$ of $f$ as well, which completes our proof.

\end{proof}

From these inputs we may deduce that the big automorphisms of the splitting field $K$ of $x^d-x^{d-1}-1$ (or of its irreducible factor in the case it is reducible) generate its full Galois group for $d \ge 6$.

Indeed, if $d \ge 6$ and $d\neq 5 \mod 6$ there are at least two roots $\alpha_1$ and $\alpha_2$ outside the unit disk. But the set of permutations that fix $\alpha_1$ or $\alpha_2$, or send one to another generate $S_d$:  we can generate any permutation by first flipping $\alpha_1$ with its final position (if it goes to something other than $\alpha_2$ this automorphism is big because it fixes $\alpha_2$, and if it goes to $\alpha_2$ it is big because it takes $\alpha_1$ to $\alpha_2$), and once $\alpha_1$ is in its target position we can the rest as necessary (this is a big automorphism because it fixes $\alpha_1$).

In the case $d \equiv 5 \mod 6$ with $d\ge 6$ we get at least two roots of absolute value strictly larger than $1$, which must both be roots of the irreducible factor (since $x^2-x+1$ has roots of absolute value $1$). The same argument as in the $d\neq 5 \mod 6$ for the roots of the irreducible factor shows that we the big automorphisms generate $S_{d-2}$ subgroup of the Galois group. With this we may conclude:

\begin{proposition}

Let $d\ge 6$ and let $a_n$ be the $d$-th generalization of the Narayana sequence. If $\alpha$ is a real number with $\|\alpha a_n\| \rightarrow 0$, then $\alpha$ is an integer.

\end{proposition}

\begin{proof}

In the case $d\neq 5 \mod 6$ we simply apply Theorem \ref{alpharecurence} to obtain that $\alpha$ belongs to the splitting field $K$ of $x^d-x^{d-1}-1$ and that $\sigma(\alpha) = \alpha$ for each big automorphism of $\alpha$, which implies $\alpha \in \mathbb{Q}$ since we argued these automorphisms generate $Aut(K/\mathbb{Q})$. And $\alpha$ actually has to be an integer, because if $\alpha$ is a rational with denominator $q$ and $\|\alpha a_n\|  \rightarrow 0$ this implies that eventually every term of the sequence $a_n$ is a multiple of $q$. Working backwards, this implies that every term of the sequence is a multiple of $q$, including $a_1=1$. It follows $\alpha \in \mathbb{Z}$.

In the case $d = 5 \mod 6$, notice that the roots of the factor $(x^2-x+1)$ each satisfy $x^3+1 = 0$. So the sequence $b_n = (a_n+a_{n+3})$ neutralizes those roots, hence it satisfies a linear recurrence whose characteristic polynomial is the remaining factor, which is irreducible. The $b_n$ are also non-zero since the $a_n$ are eventually positive. We may then apply Theorem \ref{alpharecurence} to $b_n$, which must also satisfy $\|\alpha b_n\|\rightarrow 0$ , which yields $\sigma(\alpha) = \alpha$ for each big automorphism of its splitting field. Again, since they generate the whole automorphism group this implies $\alpha \in \mathbb{Q}$, and then $\alpha \in \mathbb{Z}$.

\end{proof}

With this lemma it follows that if $\alpha$ is irrational, then for each integer $d \neq 0$ sequence $\|d \alpha a_i \|$ doesn't converge to zero. An application of Theorem \ref{uniform} completes the proof.

\printbibliography{}

\textit{E-mail address:} \texttt{rsangelo@stanford.edu}

\end{document}